\documentclass[11pt,twoside]{amsart}
\usepackage[
backend=biber,
style=alphabetic,
sorting=ynt
]{biblatex}
\usepackage[a4paper]{geometry}
\usepackage{amstext}
\usepackage{amsthm}
\usepackage{amssymb}
\usepackage{stmaryrd}
\usepackage[unicode=true,pdfusetitle,
 bookmarks=true,bookmarksnumbered=false,bookmarksopen=false,
 breaklinks=false,pdfborder={0 0 1},backref=false,colorlinks=false]
 {hyperref}
 \usepackage{cleveref}
 \usepackage{placeins}
 \usepackage{tikz-cd}

\makeatletter

\numberwithin{equation}{section}
\numberwithin{figure}{section}
\theoremstyle{plain}
\newtheorem*{theorem*}{\protect\theoremname}
\newtheorem*{corollary*}{\protect\corollaryname}
\newtheorem*{lemma*}{\protect\lemmaname}
\theoremstyle{plain}
\newtheorem{theorem}{\protect\theoremname}
\theoremstyle{plain}
\newtheorem{lemma}{\protect\lemmaname}
\theoremstyle{plain}
\newtheorem{proposition}[theorem]{\protect\propositionname}
\theoremstyle{remark}
\newtheorem{remark}{\protect\remarkname}
\theoremstyle{plain}
\newtheorem{corollary}{\protect\corollaryname}
\theoremstyle{definition}
\newtheorem{definition}{\protect\definitionname}
\theoremstyle{definition}

\theoremstyle{plain}
\newtheorem{assumption}{\protect\assumptionname}
\theoremstyle{plain}
\newtheorem{conjecture}{Conjecture}
\theoremstyle{definition}
\newtheorem{construction}
{Construction}
\theoremstyle{definition}
\newtheorem{notation}
{Notation}

\usepackage{amssymb,amsthm,amsmath,amsfonts,amscd, mathtools}
\usepackage{bm}
\usepackage{mathrsfs}
\usepackage{graphicx}
\usepackage{url}
\usepackage{color}
\usepackage[active]{srcltx}
\usepackage[matrix,arrow]{xy}
\usepackage{mathrsfs}
\usepackage{enumerate}
\usepackage{amsopn}
\usepackage{bbm}
\usepackage[normalem]{ulem}
\usepackage{hyperref}
\allowdisplaybreaks[1]
\usepackage[english]{babel}
\usepackage{bbm}
\usepackage[acronym, nonumberlist]{glossaries}
\usepackage{mhchem}
\usepackage{enumitem}
\date{\today}

\newcommand{\ddi}[2]{\frac{\partial #1}{\partial #2}}

\newcommand{\kl}{\left(}

\newcommand{\kr}{\right)}
\newcommand{\el}{\left[}
\newcommand{\estMLMC}{\mathcal{A}_{\text{MLMC}}}
\newcommand{\estMC}{\mathcal{A}_{\text{MC}}}

\newcommand{\er}{\right]}
\newcommand{\R}{\mathbb{R}}
\newcommand{\Z}{\mathbb{Z}}
\newcommand{\N}{\mathbb{N}}

\newcommand{\Ocal}{\mathcal{O}}

\newcommand{\E}{\mathbb{E}}

\renewcommand{\d}{\mathrm{d}}

\newcommand{\var}{\text{Var}}
\newcommand{\work}{\text{Work}}
\newcommand{\Var}{\text{Var}}

\newcommand{\TOL}{\text{TOL}}
\newcommand{\tol}{\text{TOL}}

\newcommand{\norm}[1]{\left\lVert#1\right\rVert}

\let\P\relax
\newcommand{\P}{\mathbb{P}}

\renewcommand{\phi}{\varphi}

\newcommand{\vast}{\bBigg@{3}}
\newcommand{\Vast}{\bBigg@{5}}

\makeatother

\providecommand{\corollaryname}{Corollary}
\providecommand{\definitionname}{Definition}
\providecommand{\examplename}{Example}
\providecommand{\lemmaname}{Lemma}
\providecommand{\remarkname}{Remark}
\providecommand{\theoremname}{Theorem}
\providecommand{\assumptionname}{Assumption}
\providecommand{\propositionname}{Proposition}

\def\@setthanks{\vspace{-\baselineskip}\def\thanks##1{\@par##1\@addpunct.}\thankses}
\makeatother

\begin{document}

\title[Forward Propagation of Low Discrepancy Through McKean--Vlasov Dynamics]{Forward Propagation of Low Discrepancy Through McKean--Vlasov Dynamics: From QMC to MLQMC}

\begin{abstract}
This work develops a particle system addressing the approximation of \linebreak McKean--Vlasov stochastic differential equations (SDEs). The novelty of the approach lies in involving low discrepancy sequences nontrivially in the construction of a particle system with coupled noise and initial conditions. Weak convergence for SDEs with additive noise is proven. A numerical study demonstrates that the novel approach presented here doubles the respective convergence rates for weak and strong approximation of the mean-field limit, compared with the standard particle system. These rates are proven in the simplified setting of a mean-field ordinary differential equation in terms of appropriate bounds involving the star discrepancy for low discrepancy sequences with a group structure, such as Rank-1 lattice points. This construction nontrivially provides an antithetic multilevel quasi-Monte Carlo estimator. An asymptotic error analysis reveals that the proposed approach outperforms methods based on the classic particle system with independent initial conditions and noise.

\end{abstract}

\author{Nadhir Ben Rached}
\address[Nadhir Ben Rached]{Department of Statistics, School of Mathematics, University of Leeds, UK}
\email{N.BenRached@leeds.ac.uk}

\author{Abdul-Lateef Haji-Ali}
\address[Abdul-Lateef Haji-Ali]{Department of Actuarial Mathematics and Statistics, School of Mathematical and Computer Sciences, Heriot-Watt University, Edinburgh, UK}
\email{a.hajiali@hw.ac.uk}

\author{Raúl Tempone}
\address[Raúl Tempone]{Computer, Electrical and Mathematical Sciences \& Engineering Division (CEMSE), King Abdullah University of Science and Technology (KAUST), Thuwal, Saudi Arabia and Alexander von Humboldt Professor in Mathematics for Uncertainty Quantification, RWTH Aachen University, Aachen, Germany}
\email{raul.tempone@kaust.edu.sa, tempone@uq.rwth-aachen.de}

\author{Leon Wilkosz$^\ast$}
\address[Leon Wilkosz]{Computer, Electrical and Mathematical Sciences \& Engineering Division (CEMSE), King Abdullah University of Science and Technology (KAUST), Thuwal, Saudi Arabia} 
\email{leon.wilkosz@kaust.edu.sa}
\thanks{Corresponding author$^\ast$}

\maketitle
\tableofcontents

	\section{Introduction}
	\label{section:introduction}
This work approximates the expectation of certain functionals applied to quantities that depend on an underlying McKean--Vlasov dynamic. These equations naturally appear in many areas as they are critical in socio-economic modeling for pedestrian flow \cite{MAHATO2018} and opinion dynamics \cite{Chaintron2022} (for their application in physics, more specifically in kinetic theory, see \cite{Chaintron2022}). Moreover, McKean--Vlasov equations appear in the calibration of financial models (e.g., see \cite{bayer2022rkhs}).
 
A McKean--Vlasov equation can be described as a stochastic differential equation (SDE) whose drift and diffusion depend on the current state and time and the law of the solution at time $t$. Let $\xi$ be a random variable, $a, \sigma, \kappa_1, \kappa_2: \R \times \R \to \R$ be functions, and $W$ be a Wiener process. Then, the dynamic can be written as follows:
\begin{align} \label{eq:McKean--Vlasov}
\begin{split}
Z(t) =  \xi & + \int_0^t a\left(Z(s), \int_{\R} \kappa_1(Z(s), z) \mathcal{L}(Z(s))(d z) \right) d s \\
& + \int_0^t \sigma\left(Z(s), \int_{\R} \kappa_2(Z(s), z) \mathcal{L}(Z(s))(dz) \right) d W(s),
\end{split} \; ,t \in [0, T],
\end{align}
where $\mathcal{L}(\cdot)$ denotes the law of the respective variable. Various works have proven the existence and uniqueness of strong solutions, one of them being \cite{hajiali2021simple}. One way to approximate this type of equation is to view it as the mean-field limit of a coupled particle system. Convergence results of this kind were shown, among others, in \cite{hajiali2021simple}. Approximating the law dependence as a discrete measure of particles, the dynamics of the particle system in question take the form
\begin{align}\label{eq:Particle-System}
\begin{split}
X_i^P(t) =  \xi_i & + \int_0^t a\left(X_i^P(s), \frac{1}{P} \sum_{j = 1}^P \kappa_1(X_i^P(s), X_j^P(s)) \right) d s \\
& + \int_0^t \sigma\left(X_i^P(s), \frac{1}{P} \sum_{j = 1}^P \kappa_2(X_i^P(s), X_j^P(s)) \right) d W_i(s),
\end{split}, i = 1, \hdots, P, \; t \in [0, T],
\end{align}
where $\xi_i, W_i$ are independent copies of $\xi, W$. Passing to the limit as $P$ goes to infinity, each particle $X_i^P$ converges in the strong sense, i.e. in $L^2$, and weak sense to the limit process $Z_i$, where $Z_i$ is driven by initial condition $\xi_i$ and noise $W_i$ and follows the corresponding McKean--Vlasov dynamics, \cite{hajiali2021simple}. Another approach is solving the Fokker-Plank equation, a nonlinear PDE, for the probability density function of $Z(t)$. The advantage of particle systems is that one can simulate them directly with classic schemes for SDEs for a large enough number of particles $P$ to satisfy given accuracy requirements. Let $(\Omega, \P, \mathcal{F})$ be the underlying probability space and $C_T = C([0, T], \R)$. A different angle on the McKean--Vlasov equation allows us to view it in the weak sense as a mapping of probability measures
$$
\Phi: \mathcal{P}(\R \times C_T) \to \mathcal{P}(C_T), \; \mathcal{L}(\xi, W) \longmapsto \mathcal{L}(Z), 
$$
$\mathcal{P}(\cdot)$ denoting the set of probability measures on the respective space. This map depends on the drift and diffusion coefficients $a, \sigma$. Let $W_1, \hdots W_P, \; \xi_1, \hdots, \xi_P$ be the noise and initial conditions of the above particle system. Let us assume that
\[
\mu : \Omega \to \mathcal{P}(\R \times C_T)
\]
is a random measure with
\[
\mu(\omega) = \frac{1}{P} \sum_{i = 1}^{P} \delta_{\xi_i(\omega), W_i(\omega)}.
\]
Then the joined law of particles on $C_T$,
\[
L^P(\bm X^P(\omega)) = \frac{1}{P} \sum_{i = 1}^{P} \delta_{X^P_i(\omega)}
\]
is exactly $L^P(\bm X^P(\omega)) = \Phi(\mu(\omega))$. Writing $\Phi_T$ for the projection to time $T$, we find in particular, for any integrable function $g:\R \to \R$,
$$
\E\left[\frac{1}{P} \sum_{j = 1}^P g(X_j^P(T))\right] = \E\left[ \int g \; d \Phi_T(\mu(\omega)) \right],
$$
so the particle system is in a sense a McKean--Vlasov equation with a special choice of driving signals. Note that the above interpretation however requires a pathwise interpretation of the SDE, which is provided by rough path theory in the case of noise with infinite $1$-variation. The map $\Phi$ has been shown to respect weak convergence, \cite{friz2020}. The idea we advocate for in this work is to replace the empirical law of i.i.d. processes $\xi_i, W_i$ by the empirical law of coupled processes $\xi_i^C, W_i^C$, where each pair arises from one point of a low discrepancy sequence. Determining a Brownian Motion completely by a quasi-Monte Carlo (QMC) point requires an appropriate map
\[
\mathcal{T} : [0, 1]^{\infty} \to C_T,
\]
the canonical choice being the Levy-Ciesielsky construction, \cite{karatzas91}, which we will use here as well. Our estimator $\bm X^{E, P, D}$, is an Euler-Maruyama approximation of the particle system with input signals $\xi_i^C, W_i^C$, where $P$ is the number of particles and $D$ the corresponding dissection of $[0, T]$. We denote by $|D|$ the maximal step size in $D$. Numerical examples demonstrate not only convergence but also improved rates for weak error and variance. To be more precise, we find numerically that
\begin{align*}
&\left| \E\left[ g(Z(T)) - \frac{1}{P} \sum_{i = 1}^P g(X^{E, P, D}_i(T))\right] \right| \approx \Ocal(P^{-2}), \\
&\var\left[ \frac{1}{P} \sum_{i = 1}^P g(X^{E, P, D}_i(T))\right] \approx \Ocal(P^{-2}),
\end{align*}
supposing the time grid is tight enough. The respective rates are both $\Ocal(P^{-1})$ in the case of a classic particle system with i.i.d. noise and initial conditions, which was proven in \cite{hajiali2017multilevel} and \cite{hajiali2021simple} for drift, diffusion coefficients and kernels that are three times continuously differentiable and bounded with bounded derivatives. Since QMC Quadrature for arbitrary continuous functions is intractable on the Hilbert cube, a counter-example is found in \cite{sobol1998103}, the continuity of the map $\Phi$ does not cover the mean-field convergence of this particular system trivially. One of the main results of this work, \Cref{thm:weak_convergence}, proves weak convergence in the case of additive noise in the following sense,

\begin{theorem*}
    Let $g: \R \to \R$ be a bounded, globally Lipschitz continuous function. Then, for any $\varepsilon > 0$, there exists a $\delta > 0$, such that for any dissection $D$ with $|D| < \delta$, we find $\tilde P(\varepsilon, D) \in \N$ such that for any $P \geq \tilde P$ we have
    \begin{align*}
        \left| \E\left[ g(Z(T)) - \frac{1}{P} \sum_{i = 1}^P g(X^{E, P, D}_i(T))\right] \right| \leq \varepsilon.
    \end{align*}
\end{theorem*}

The proof uses a workaround with Wong-Zakai approximations to avoid integration on the Hilbert cube. The key observation is here, that the Levy-Ciesielsky construction is a piece-wise linear approximation of a Brownian Motion, relying on a finite number of Gaussian variables. The pipeline is to show first, that our estimator satisfies a $P$-uniform error bound when compared to the Wong-Zakai approximations of the same particle system with respect to piecewise linear noise paths on the same dissection. These Wong-Zakai approximations are indeed functions of finite dimensional points from a low discrepancy sequence, hence weak convergence is easier to handle. This is in the end combined with topological arguments and the weak continuity of $\Phi$. The informal diagram below illustrates the strategy.

\vspace{1cm}

\begin{tikzcd}
{\bm{X}^{E, P, D}}                                            & {} \arrow[rrr]                          &  &  & {} & Z                                      \\
{} \arrow[ddd, "{$|D|$ \small{small enough},\\  $P$\small{-uniformly}}"'align=left, dashed] &                                         &  &  &    & {}                                     \\
                                                              &                                         &  &  &    &                                        \\
                                                              &                                         &  &  &    &                                        \\
{}                                                            &                                         &  &  &    & {} \arrow[uuu, "|D| \to 0"', dashed]   \\
{\bm X^{P, D} \sim \Phi\left(\frac{1}{P} \sum_{i = 1}^P \delta_{\xi_i, W_i^D} \right)}                & {} \arrow[rrr, "P \to \infty"', dashed] &  &  & {} & {Z^D \sim \Phi(\mathcal{L}(\xi, W^D))}
\end{tikzcd}
\vspace{1cm}

Unfortunately, the inverse transformation of the Gaussian distribution function occurring in the mapping $\mathcal{T}$ blocks the way to apply the famous inequality of Koksma and Hlawka. This is why we migrate to the simpler setting of a Mean-Field ordinary differential equation (ODE) (i.e. zero-diffusion McKean--Vlasov) to prove these rates in terms of bounds involving the star discrepancy
\[
\Delta_{P, N}^* := \sup_{x \in [0, 1]^N} \left|\frac{1}{P} \sum_{i = 1}^P 1_{[0, x]}(\zeta^i) - \lambda ([0, x]) \right|,
\]
in this work often denoted by $\Delta_{P, N}^* = \Delta_{\zeta}$, indicating the associated sequence of points. The second main result of this work is a consequence of \Cref{thm:weak_error} and \Cref{thm:variance}. It makes assumptions on a special choice of initial conditions and on the regularity of drift and kernel which are explained in \Cref{subsec:rates} in detail. The bottleneck for these assumptions is, as mentioned before, the application of Koksma-Hlawka.

\begin{theorem*}
Let $X^P = \{X_i^P\}_{i= 1}^P$ be the solution to the particle system with coupled initial conditions $\xi^C_i, \; i = 1, \hdots, P$ and let $Z = \{Z_i\}_{i = 1}^P$ be solutions to the Mean Field ODEs with identic initial conditions for a sequence of points $\zeta^1, \hdots, \zeta^P \in [0, 1]$ forming a finite subgroup of $\R / \Z$. Moreover, let $g : \R \to \R$ be bounded and twice continuously differentiable with bounded derivatives. Then the weak error and variance satisfy
\begin{align*}
    & \left| \E\left[\frac{1}{P} \sum_{i = 1}^P g(X^P(T))- g(Z(T)) \right] \right| \lesssim \sup_{x \in [0, 1]} \Delta_{\{\zeta + x\}}^2, \\
    & \Var\left[ \frac{1}{P} \sum_{i = 1}^P g(X_i^P(T)) \right] \lesssim \sup_{x \in [0, 1]}  \Delta_{\{\zeta + \tilde x\}}^2,
\end{align*}
where the constant does not depend on $P$.
\end{theorem*}

To provide intuition on the assumption of a group structure in the points, we remark that this restores the particles to be at least identically distributed, something not to be expected for general low discrepancy sequences. This property is for instance satisfied by rank-1 lattice points, \cite{niederreiterQMC}.

We further extend this approach to a hierarchical setting, inspired by the multilevel approximations in \cite{hajiali2017multilevel}. The Rank-1 lattice points play once more an important role here, since splitting these sequences in a certain way projects them, after random shifting, in distribution to the next smaller set of Rank-1 lattice points. The significance of this comes to light when constructing a multilevel estimator as a telescopic sum without introducing additional bias. In that regard, we would like to point the reader to the fact, that the distribution of our particle system depends heavily on the choice of points. Based on the theoretical and numerical results, we carry out an asymptotic error and work analysis for the approximation of $\E[g(Z(T))]$, yielding $\Ocal(\TOL^{-3})$ in the single level and $\Ocal(\TOL^{-2})$ in the hierarchical case. For the standard particle system with i.i.d. noise, one has $\Ocal(\TOL^{-4})$ with crude Monte Carlo and $\Ocal(\TOL^{-3})$ for multilevel Monte Carlo, \cite{hajiali2017multilevel}. Numerical examples showcasing our findings are performed for an Ornstein-Uhlenbeck-type equation and the well-known Kuramoto-Oscillator, the former having the advantage of easy-to-compute analytic solutions.

\subsection*{Outline} The objective of the second section is to introduce our single-level estimator in detail. We formulate the mapping from QMC points to Wiener paths and define the approximation. The third section contains our propagation-of-chaos-type results for McKean--Vlasov SDEs with additive noise and mean field ODEs. In the fourth section, we introduce an antithetic multilevel construction. Next, we perform an asymptotic complexity analysis for reaching a prescribed tolerance $\tol$, underpinned by our theoretical findings and the convergence rates we observe numerically. This is done for single as well as multilevel variants. Section 6 contains the numerical results.

\FloatBarrier
\section{A Novel Particle System Construction}\label{sec:construct}
\FloatBarrier

\subsection{Motivation on Higher Order Schemes}
This section provides a motivational example of why we need more sophisticated technology such as QMC to construct higher-order schemes. Having in mind the convergence analysis of numerical methods for ODEs, one might be tempted to insinuate a similar error structure for a particle system approximation. We shall allow ourselves to carry this thought forward in the following discussion, derive some consequences, and convince ourselves by means of a simple example that this kind of thinking does not apply in a trivial way to mean-field approximations. Since this serves motivational purposes, we keep the discussion less technical.

Let us assume that we have an estimator $I_P$ of an observable with underlying McKean--Vlasov dynamics that relies on the simulation of a $P$-particle system approximation. Let $I$ denote the exact quantity. Moreover, we assume for now that there exists a series expansions of the form
\[
I_P - I = \sum_{i = 1}^{\infty} \frac{C_i}{P^i}.
\]
Then we could apply here the concept of Richardson extrapolation to find a scheme of second order by defining
\[
\prescript{Ext}{}{I}_P := 2I_{2P} - I_P.
\]
We use the classic Monte Carlo estimator $\estMC(M, P, N)$ for the particle system with an Euler-Maruyama discretization with extrapolation in the particle domain to approximate the second moment of the Ornstein Uhlenbeck McKean--Vlasov SDE introduced in \Cref{subsec:OU-SDE} at some final time. The result can be seen in \Cref{fig:extrapolation}

\begin{figure}
\includegraphics[width=0.87\textwidth]{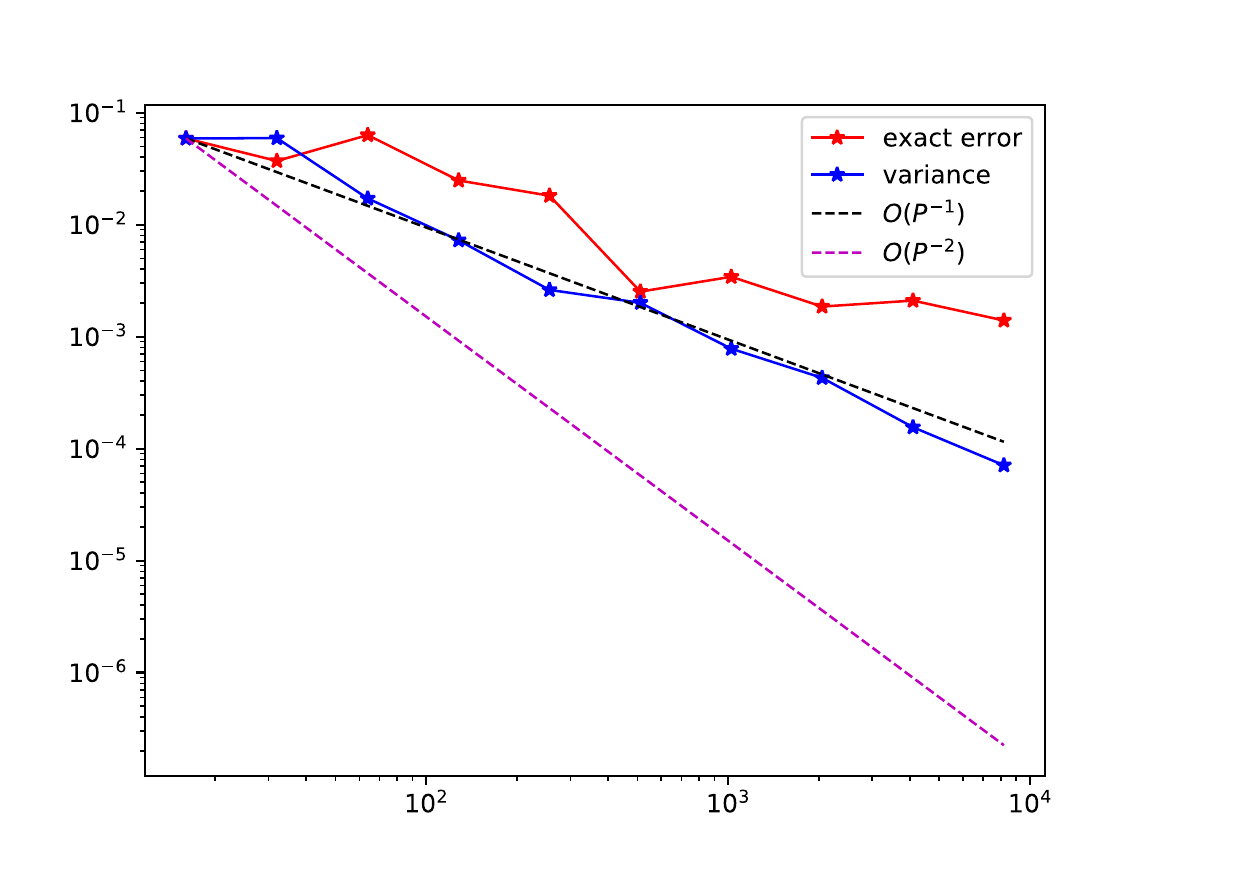}
\caption{Simulation of variance and exact error of the scheme $\prescript{Ext}{}{I}_P$.}
\label{fig:extrapolation}
\end{figure}

Contrary to the assumption on the error structure, we observe a first-order rate for variance and exact error, which validates that a series expansion in terms of the size of the particle system does not exist in this already simple case. This makes constructing more efficient schemes a challenging task that we will tackle using the benefits of QMC.

\subsection{QMC and SDEs}
A key aspect in combining QMC with the simulation of SDEs is to overcome the obstacle of making the integration domain fit -- QMC sequences are built for integration over the unit cube $[0, 1]^d$ whereas computing the expectation of observables evaluating on solutions of SDEs usually boils down to integration on $\R^d$. In the context of discretizations, it suffices to map a ($d$-dimensional) QMC point to $d$ independent Gaussian increments $\Delta W$, so essentially sending a uniform random variable in $[0, 1]$ to a Gaussian random variable. The standard way of doing so is to apply the inverse CDF method.

Suppose that we have a QMC point $\zeta = (x_1, \hdots, x_d)$, a random shift $U = (u_1, \hdots, u_d)$ where the $u_i$ are i.i.d. and uniformly distributed on $[0, 1]$, and we aim to build $d$ independent Wiener increments with step size $\Delta t$. The canonical choice, which is used frequently in the context of QMC, would now be to use the (pointwise) mapping
\[
\zeta + U \mapsto \kl \sqrt{\Delta t} \cdot \phi^{-1}\kl \{x_1 + u_1\} \kr , \hdots, \sqrt{\Delta t} \cdot \phi^{-1}\kl \{x_d + u_d\} \kr \kr,
\]
where $\{\cdot\}$ denotes the fractional part and $\phi$ the CDF of a normal random variable.
However, it has been observed empirically that this is not a good choice for building Wiener paths for SDEs. To provide the reader with some intuition about why that is, we recall the classic result from discrepancy theory that QMC approximations can converge with order $\Ocal \kl \frac{\log(N)^d}{N} \kr$ as pointed out in \cite{QMCguide}. Due to the logarithmic term, the convergence suffers in high dimension, and for SDEs, we are usually using discretizations that lie in the critical regime for QMC. Nevertheless, there is a workaround using Brownian bridge as suggested in \cite{caflisch97} and \cite{GilesQMC}. We continue with a precise construction of a mapping closely aligned to those ideas. We will denote by $\{ x \} = x \mod 1$ the fractional part of $x$. This is necessary to map randomized points back to the unit cube. If applied to multidimensional vectors, it is always meant component-wise.

\begin{construction}\label{construct:BB}
Let $\phi$ be the CDF of the standard Gaussian distribution. Let $n = 2^d \in \N$ and let $\zeta \in [0, 1]^n$ be a QMC point. Let $T > 0$ and $\Delta t := \frac{T}{n}$. Let $U$ be a $n$-dimensional vector of i.i.d. uniform random variables in $[0, 1]$. We define a stochastic process $T^{(n)}(\zeta, U) = W^{(n)}(t)$ on $[0, T]$ by the following recursive procedure.
\begin{enumerate}
\item $W^{(n)}(0) := 0$ and $W^{(n)}(T) := \sqrt{T} \cdot \phi^{-1}\kl \{\zeta^1 + U_1 \}\kr$. Set $k := 1$, $t_0 := 0, t_1 := T$.
\item For the time steps $t_0, \hdots, t_k$ there is already a path value $W^n(t_k), 0 \leq i \leq k$ assigned. These times are however not necessarily increasingly ordered. Let $0 = t^{(0)} < \hdots < t^{(k)} = T$ be a reordering of the very same times. Then, for any $1 \leq i \leq k$ we define
\begin{align*}
t_{k + i} & := \frac{t^{(i)} + t^{(i - 1)}}{2}, \\
W^{(n)}(t_{k + i}) &:= \mu_i + \sigma_i \cdot \phi^{-1}\kl \{\zeta^{k + i} + U_{k + i} \}\kr,
\end{align*}
where
\begin{align*}
\mu_i & = W^{(n)}(t^{(i)}) + \frac{1}{2} \kl W^{(n)}(t^{(i + 1)}) - W^{(n)}(t^{(i)}) \kr, \\
\sigma_i & = \frac{1}{2} \sqrt{t^{(i + 1)} - t^{(i)}}.
\end{align*}
Now increment $k := 2k$ and continue with this step until $k = n$.
\item We interpolate linearly between the previously defined values.
\end{enumerate}
\end{construction}

A new component that is inserted in $[t^{(i)}, t^{(i + 1)}]$ only affects the values of the path inside that interval, whereas in the naive incremental mapping, the domain of influence is much larger. Overall this leads to a reduction of the so-called "effective dimension" as pointed out in \cite{caflisch97}.

It is a direct consequence of (\hspace{1sp}\cite{karatzas91}, pp.57-58) that the process $\{W^{(n)}(t, \omega)\}$ from \Cref{construct:BB} converges uniformly in $t$ to a Wiener process, almost surely. The following definition is convenient and allows us to identify our estimator, to be defined in the next subsection, with an Euler-Maruyama approximation of a continuous-time particle system with coupled noise.

\begin{definition}[Continuous paths from infinite dimensional sequences]
 Let $\zeta^1, \zeta^2, \hdots \in [0, 1]^{\infty}$ a sequence of infinite dimensional QMC points. Examples are an infinite dimensional version of Halton points, or rank-1 lattice points with an infinite dimensional generating vector. Let $U$ be a random vector in $[0, 1]^{\infty}$ of i.i.d. uniform random variables in $[0, 1]$. Let $U_{k:n}, \; \zeta^j_{k:n}$ be the vectors containing all components form $k$ to $n$ for $k \leq n$ of $U$ and $\zeta^j$, respectively. Let $T^{(n)}(\zeta^i_{1:n}, U_{1:n})$ be the process from \Cref{construct:BB}. We define the $i$-th correlated Wiener processes
\[
W^C_i := T(\zeta^i, U) := \lim_{n \to \infty} T^{(n)}(\zeta^i_{1:i}, U_{1:i}).
\]
\end{definition}

\subsection{Estimator Construction}

In this section, we develop the actual approach to approximating \Cref{eq:McKean--Vlasov}.

Let $D = \{0 = t_1 < \hdots < t_N = T\}$ be a dissection of the time interval $[0, T]$. For simplicity, we assume uniform time steps. Denote by $X_{i, n}^P(\xi, W), \; 1 \leq n \leq N$ an Euler-Maruyama approximation of \Cref{eq:Particle-System} on $D$ with input signals $(\xi_i, W_i), \; 1 \leq i \leq P$. The next definition gives a description of the paths we will use subsequently.

\begin{definition}\label{def:qmc_measure}
Let $N$ be any power of $2$. Let $\zeta^1, \zeta^2, \hdots \subset [0, 1]^{N + 1}$ be a sequence of QMC points and let $U$ be an $(N + 1)$-dimensional random vector of i.i.d. uniformly distributed random variables in $[0, 1]$. Let $D = \{0 = t_0 < \hdots < t_N \}$ be a dissection of $[0, T]$. We assume that there exists a map $\mathcal{H}: \R \to \R$ such that $\mathcal{H}$ applied to a uniform random variable in $[0, 1]$ carries the same distribution as $\xi$. Let $T^{(N)}$ be the mapping defined in \Cref{construct:BB}. We define random variables
\begin{enumerate}
\item $W^{(N)}_p := T^{(N)}(\zeta^p, U_{2:N + 1})$,
\item $\xi_p^C := \mathcal{H}(\{\zeta^p_1 + U_1\})$.
\end{enumerate}
\end{definition}

Finally, we can describe our estimator for an observable evaluating on the solution of the McKean--Vlasov equation.

\begin{definition}\label{def:single-level-estimator}
Let $g: \R \to \R$ be some continuously differentiable function. Let $\zeta^1, \hdots, \zeta^P \subset [0, 1]^{N + 1}$ be a sequence of points and let $U^{1}, \hdots, U^{M}$ be $M$ i.i.d. samples from an $(N + 1)$-dimensional uniform random shift. Let $D = \{0 = t_0 < \hdots < t_N \}$ be a dissection of $[0, T]$. We define the increments
\begin{align*}
& \Delta W_{p, n}^{(N), i} := W_{p}^{(N)}(t_{n + 1}, U^i_{2:N + 1}) - W_{p}^{(N)}(t_n, U^i_{2:N + 1}), \\
& \Delta t_n := t_{n + 1} - t_n.
\end{align*}
We define iteratively
\begin{align*}
X_{p, n + 1}^{P, i} &:= X_{p, n}^{P, i}  + a\kl X_{p, n}^{P, i}, \frac{1}{P} \sum_{j = 1}^P \kappa_1(X_{p, n}^{P, i}, X_{j, n}^{P, i}) \kr \Delta t_n & \\
&  \hspace{1.5cm} + \sigma \kl X_{p, n}^{P, i}, \frac{1}{P} \sum_{j = 1}^P \kappa_2(X_{p, n}^{P, i}, X_{j, n}^{P, i}) \kr \Delta W_{p, n}^{(N), i},& \\
X_{p, 0}^{P, i} &:= \xi_i^C 
\end{align*}
The final estimator is given by
\begin{align*}
\mathcal{A}(M, P, N) := \frac{1}{MP} \sum_{i = 1}^M \sum_{p = 1}^P g(X_{p, N}^{P, i}).
\end{align*}
\end{definition}

\begin{remark}
In other words, one can say that the estimator defined in \Cref{def:single-level-estimator} is the Euler-Maruyama approximation $X_{i, n}^P(\xi, W^C), \; 1 \leq n \leq N$. It follows from classic stochastic calculus that this discretization converges indeed to the solution of the respective continuous-time system of equations and that this system is well-posed.
\end{remark}

\FloatBarrier

\section{Proof of Convergence and Rates}

\subsection{Mean-Field Convergence}

In this section, we will prove weak convergence of the estimator in \Cref{def:single-level-estimator} for McKean--Vlasov equations with additive noise of the type
\begin{align}\label{additive_noise_MVSDE}
\begin{split}
Z(t) =  \xi & + \int_0^t a\left(Z(s), \int_{\R} \kappa_1(Z(s), z) \mathcal{L}(Z(s))(d z) \right) ds + \sigma W(t), \; t \in [0, T]
\end{split}
\end{align}
where $\sigma \in \R$ is  a constant, $W$ is a Wiener process, $\xi$ a random variable, and $a, \kappa$ are functions with properties specified below. To simplify the notation when working with particle systems, we write
\[
\bm{a}(\bm{x}) = \left( a\left(x_1, \frac{1}{P} \sum_{i = 1}^P \kappa(x_1, x_j)\right), \hdots, a\left(x_P, \frac{1}{P} \sum_{i = 1}^P \kappa(x_P, x_j)\right) \right), \; \bm{x} \in \R^P.
\]
In the following, $\bm{W} = (W_1, \hdots, W_P)$ will be a vector of Wiener processes without further specified covariance unless specified further. Similarly, without further specification, $\bm{\xi} = (\xi_1, \hdots, \xi_P)$ will denote an arbitrary random vector with $\xi_i \sim \mathcal{L}(\xi), \; 1 \leq i \leq P$. Moreover, given a dissection $D = \{0 = t_0 < \hdots < t_{\#D - 1} = T\}$ of an interval $[0, T]$, we denote by $W^D$ the piece-wise linear approximation of $W$ corresponding to $D$ in the usual way. We will denote the maximal step size by $|D|$. Thereafter, the auxiliary objects we consider are equations
\[
\bm{X}^D(t) = \bm{X}^D(t; 0, \bm{\xi}) = \bm{\xi} + \int_0^t \bm{a}(\bm{X}^D(s)) ds + \sigma \bm{W}^D(t),
\]
which can be described best as the Wong-Zakai approximation of the particle system with initial data $\bm{\xi}, \bm{W}$. These are a class of random $P$-dimensional ODEs for any $P \in \N$ which can be solved pathwise. The reader may notice that the notation gives no information about the dimension of the particle system. However, this slight abuse of notation is justified in the sense that the subsequent results hold $P$-uniformly, i.e. for all system sizes $P \in \N$. Moreover, the discrete object rising from an Euler-Maruyama approximation of the same equation corresponding to the same dissection will be denoted by $\bm{X}^E(t_i), \; 0 \leq i \leq \#D - 1$. We make some regularity assumptions that will be sufficient for the subsequent discussion. They shall hold throughout this very subsection.

\begin{assumption}\label{ass:weak_proof}
    \begin{enumerate}
        \item[A1)] $a: \R \times \R \to \R$ is a bounded function that is once continuously differentiable; moreover, let these derivatives be bounded.
        \item[A2)] $\xi$ is a random variable such that there exists a continuous transformation $\mathcal{H}:[0, 1] \to \R$, such that $\lambda \circ \mathcal{H}^{-1} = \mathcal{L}(\xi)$
        \item[A3)] $\kappa : \R \times \R \to \R$ is a bounded function that has continuous and bounded first-order partial derivatives
        \item[A4)] $\zeta^1, \zeta^2, \hdots$ is an infinite sequence of points in $[0, 1]^{\infty}$ such that for any finite $I \subset \N$ and any Riemann integrable function $g:[0, 1]^{\#I} \to \R$, the projections $\pi_I(\zeta^j) = (\zeta^j_i)_{i \in I}, \; j = 1, 2, \hdots$ satisfy
        \begin{align*}
            \frac{1}{P} \sum_{j = 1}^P g(\pi_I(\zeta^j)) \underset{P \to \infty}{\longrightarrow} \int_{[0, 1]^{\#I}} g dx.
        \end{align*}
    \end{enumerate}
\end{assumption}

To see that \textit{A4)} is in fact realistic, we might consider infinite-dimensional Halton points, which exist since there are infinitely many primes. The projections are then just ordinary finite dimensional Halton points. Otherwise, we invite the reader to consider Rank-1 lattice points for an infinite-dimensional generating vector $z \in \Z^{\infty}$. Moreover, it is well known that Riemann integrability is the precise condition for the convergence of QMC Quadrature. We refer the reader to \cite{niedbook} for details. We formulate the main result, whose proof follows at the end of this subsection.

\begin{theorem}\label{thm:weak_convergence}
    Let $g: \R \to \R$ be a bounded, globally Lipschitz continuous function. Let $Z$ be the strong solution of \Cref{additive_noise_MVSDE}. Let $\bm X^E(T)$ be the Euler-Maruyama approximation of \Cref{additive_noise_MVSDE} as in \Cref{def:single-level-estimator} for a sequence of points satisfying \Cref{ass:weak_proof}. Then, for any $\varepsilon > 0$, there exists a $\delta > 0$, such that for any dissection $D$ with $|D| < \delta$, we find $\tilde P(\varepsilon, D) \in \N$ such that for any $P \geq \tilde P$ we have
    \begin{align*}
        \left| \E\left[ g(Z(T)) - \frac{1}{P} \sum_{i = 1}^P g(X^E_i(T))\right] \right| \leq \varepsilon.
    \end{align*}
\end{theorem}

The subsequent results are modifications of classic facts for ODEs. We need to make these adjustments to obtain a bound on the distance from the Euler-Maruyama scheme to the solution of the Wong-Zakai equation which holds regardless of the size of the particle system. In this spirit, the first result is about Lipschitz continuity in the initial condition of the Wong-Zakai solution. To simplify notation, we denote the average norm for $\bm x \in \R^n$ by
\[
\norm{\bm x}_{av} = \frac{1}{n} \sum_{i = 1}^n |x_i|.
\]

\begin{lemma}[P-Uniform Average Lipschitz Continuity in the Initial Condition]\label{lem:Lip}
    For $t_1 < t_2 \in [0, T]$, $\bm{x}_1, \bm{x}_2 \in \R^P$, with the Lipschitz constants $M$ of $a$ and $K$ of $\kappa$ with respect to $\norm{\cdot}_{av}$, we find pointwise, independent of $P$ and $D$,
    \begin{align}
        \norm{\bm{X}^D(t_2; t_1, \bm x_1) - \bm{X}^D(t_2; t_1,\bm x_2) }_{av} \leq \exp(C(t_2 - t_1)) \norm{\bm{x}_2 - \bm{x}_1}_{av},
    \end{align}
    where $C = \frac{M + MK}{2}$.
\end{lemma}

\begin{proof}

    We see that the noise term cancels immediately and obtain for $t_1 \leq t \leq t_2$,
    \begin{align*}
        \left\| \right. & \left.\bm{X}^D(t; t_1, \bm x_1) - \bm{X}^D(t; t_1,\bm x_2) \right\|_{av} \\
        & \leq \; \norm{\bm x_2 - \bm x_1}_{av} + \int_{t_1}^{t_2} \norm{ \bm{a}(\bm{X}^D(s; t_1, \bm x_1)) - \bm{a}(\bm{X}^D(s; t_1, \bm x_2))  }_{av} ds\\
        &\overset{(*)}{\leq} \; \norm{\bm x_2 - \bm x_1}_{av} + \int_{t_1}^{t_2} \frac{M + MK}{2} \norm{\bm{X}^D(s; t_1, \bm x_1) - \bm{X}^D(s; t_1, \bm x_2)  }_{av} ds.
    \end{align*}
    The assertion follows now by Gronwalls Lemma. To see $(*)$, assume $\bm{x}, \bm{y} \in \R^P$. We calculate,
    \begin{align*}
        \norm{\bm a(\bm x) - \bm a(\bm x)}_{av} &= \frac{1}{P} \sum_{i = 1}^P \left|a\left(x_i, \frac{1}{P} \sum_{j = 1}^P \kappa(x_i, x_j) \right) - a\left(y_i, \frac{1}{P} \sum_{j = 1}^P \kappa(y_i, y_j) \right)\right| \\
        & \leq \frac{M}{2P} \sum_{i = 1}^P \left(|x_i - y_i| + \left| \frac{1}{P} \sum_{j = 1}^P \left(\kappa(x_i, x_j) - \kappa(y_i, y_j) \right) \right| \right) \\
        & \leq \frac{M}{2P} \sum_{i = 1}^P \left(|x_i - y_i| + \frac{K}{2P} \sum_{j = 1}^P \left(|x_i - y_i| + |x_j - y_j| \right) \right) \\
        & \frac{M + MK}{2} \norm{\bm x - \bm y}_{av}.
    \end{align*} 
\end{proof}

The next natural step to a $P$-uniform global error is to establish a $P$-uniform local error.

\begin{lemma}[P-Uniform Local Average Error]\label{lem:loc_err}
Let $D = \{0, t\}$ be the dissection of $[0, t]$ containing only two points. Then for some $0 \leq C_1, C_2 < \infty$ independent of $P$, we have
\begin{align*}
    \E\left[\norm{ \bm{X}^D(t; 0, \bm{X}^E(0)) - \bm{X}^E(t)}_{av}\right] \leq C_1 t^2 + C_2 t^{\frac{3}{2}},
\end{align*}
independent of $P$.
\end{lemma}

\begin{proof}
    We write for convenience $\bm{x} = \bm{X}^E(0)$. Then, by the mean value theorem,
    \begin{align*}
        \norm{\bm{X}^D(t; 0, \bm{x}) - \bm{X}^E(t)}_{av} \leq & \; \int_0^t \norm{\bm{a}(\bm{X}^D(s; 0, \bm{x})) - \bm{a}(\bm{x})}_{av} ds \\
        \leq & \; t\int_0^t \norm{\frac{d}{ds} \bm{a}(\bm{X}^D(s; 0, \bm x))}_{av} ds \\
    \end{align*}
    Whenever not relevant, we leave out function arguments for clarity of the presentation. Moreover, we use $\partial_i, \; i = 1, 2$ for the partial derivative after the first and second argument of $a$, respectively. Since the noise is linear by assumption, we can differentiate $\bm X^D$. We compute for any $1 \leq i \leq P$,
    \begin{align*}
        \frac{d}{ds} \bm{a}_i(\bm{X}^D(s)) = & \; \partial_1 a \frac{d}{ds}X^D_i + \partial_2 a \frac{1}{P} \sum_{i = 1}^P \partial_1 \kappa(X^D_i, X^D_j) \frac{d}{ds} X^D_i \\
        & \; + \partial_2 a \frac{1}{P} \sum_{i = 1}^P \partial_2 \kappa(X^D_i, X^D_j) \frac{d}{ds}X^D_j.
    \end{align*}
    Further, we have for any $1 \leq i \leq P$,
    \begin{align*}
        \frac{d}{ds}X^D_i = a + \frac{1}{t}W_i(t),
    \end{align*}
    i.e. with Hölder's inequality,
    \begin{align*}
        \E[|X^D_i|] & \leq \norm{a}_{L^{\infty}} + \frac{1}{t} \E[|W_i(t)|] \\
        & \leq \norm{a}_{L^{\infty}} + \frac{1}{t} \E[W_i(t)^2]^{\frac{1}{2}} \\
        & = \norm{a}_{L^{\infty}} + t^{-\frac{1}{2}}
    \end{align*}
    Inserting this into our previous computation, we find constants $C_1(a, \kappa), C_2(a, \kappa)$, independent of $t$ or $P$, such that
    \begin{align*}
        \E\left[ \left| \frac{d}{ds} \bm{a}_i(\bm{X}^D(s)) \right| \right] \leq C_1 + C_2t^{-\frac{1}{2}}.
    \end{align*}
    We conclude,
    \begin{align*}
        \E\left[\norm{\bm{X}^D(t; 0, \bm{x} - \bm{X}^E(t)}_{av} \right] \leq & t\int_0^t \E\left[\norm{\frac{d}{ds} \bm{a}(\bm{X}^D(s; 0, x))}_{av} \right] ds \\
        \leq C_1 t^2 + C_2t^{1 + \frac{1}{2}}.
    \end{align*}
\end{proof}

Finally, we can state the desired global bound, again independent of $P$, only dependent on the dissection.

\begin{lemma}[P-Uniform Global Error]\label{lem:global_error}
    Let $W_i, \; 1 \leq i \leq \infty$ be Wiener processes without further specified covariance and $D$ be a dissection of $[0, T]$. Let $g: \R \to \R$ be globally Lipschitz with Lipschitz constant $L$. Let $W^D_i$ be a piece-wise linear approximation of the $i$-th process. Then for any $P \in \N$, for any sub-collection $W_{i_1}, \hdots, W_{i_P}$,
    \begin{align}
        \left| \E\left[ \frac{1}{P} \sum_{i = 1}^P \left(g(X_i^E(T)) - g(X_i^D(T))\right)\right] \right| \leq TL\exp(CT)(C_1 |D| + C_2 |D|^{\frac{1}{2}}).
    \end{align}
\end{lemma}

\begin{proof}
    Using \Cref{lem:Lip}, \Cref{lem:loc_err}, and that starting times are immaterial, we find
    \begin{align*}
        \left|\E\left[ \frac{1}{P} \sum_{i = 1}^P \left( \right. \right. \right. & \left. \left. g(X_i^E(T))\;  - g(X_i^D(T))\right) \right] \Bigg| \leq L \E\left[ \norm{\bm{X}^E(T) - \bm{X}^D(T; 0,\bm{X}^E(0)) }_{av} \right] \\
        & = L\E\left[\norm{\bm{X}^D(T; T,\bm{X}^E(T)) - \bm{X}^D(T; 0,\bm{X}^E(0))}_{av} \right]\\
        & =L\E\left[ \norm{ \sum_{n = 0}^{\#D - 2} \bm{X}^D(T; t_{n + 1},\bm{X}^E(t_{n + 1})) - \bm{X}^D(T; t_n,\bm{X}^E(t_n)) }_{av} \right] \\
        & \leq L \sum_{n = 0}^{\#D - 2} \E\left[\norm{\bm{X}^D(T; t_{n + 1},\bm{X}^E(t_{n + 1})) - \bm{X}^D(T; t_n,\bm{X}^E(t_n)) }_{av}\right] \\
        & \leq L\exp(CT) \sum_{k = 0}^{\#D - 2} \E\left[\norm{\bm{X}^E(t_{n + 1}) - \bm{X}^D(t_{n + 1}; t_n,\bm{X}^E(t_n))}_{av}\right] \\
        & \leq L\exp(CT)\sum_{k = 0}^{\#D - 2} (C_1 (t_{n + 1} - t_n)^2 + C_2 (t_{n + 1} - t_n)^{1 + \frac{1}{2}}) \\
        & \leq L\exp(CT)(C_1 |D| + C_2 |D|^{\frac{1}{2}}) \sum_{n = 0}^{\#D - 2} (t_{n + 1} - t_n) \\
        & = TL\exp(CT)(C_1 |D| + C_2 |D|^{\frac{1}{2}}).
    \end{align*}
\end{proof}

What follows is a discussion about the weak convergence of the herein-defined objects, i.e. the Brownian Bridge mapping and the Wong-Zakai approximation. The goal is to be able to use the powerful results from \cite{friz2020} for the proof of weak convergence. For the subsequent discussion, we denote by $C_T = C([0, T], \R)$ the space of real-valued continuous functions on $[0, T]$, equipped with the uniform topology, i.e. with the metric
\[
d_{\infty}(f, g) = \sup_{t \in [0, T]} |f(t) - g(t)|.
\]
Equipping $C_T$ with the Borel $\sigma$-field corresponding to the uniform topology, it becomes a measurable space and we write $\mu_W$ for the Wiener measure on $C_T$. Moreover, adopting notation from \cite{friz2020}, let
\[
L^P(\bm \xi, \bm W^D) = \frac{1}{P} \sum_{i = 1}^P \delta_{\zeta^i, W_i^D}
\]
be the empirical measure of the $P$ paths and initial conditions on $\R \times C_T$. For the corresponding empirical measures of particles on $C_T$, we write
\[
L^P(\bm X^D) = \frac{1}{P} \sum_{i = 1}^P \delta_{X^D_i}.
\]
We note that using the results from \cite{friz2020}, the following results translate directly to the weak convergence of the corresponding distributions of the McKean--Vlasov dynamic with mentioned noise.

\begin{lemma}\label{lem:weak-P-convergence}
    Let $\zeta^1, \zeta^2, \hdots$ be a sequence of points as in \Cref{ass:weak_proof} and let $\xi_i = \mathcal{H}(\zeta^i_1 + x_1)$, $W_i^D = T^{(n)}(\zeta^i_{2:n+1}, x_{2:n + 1})$ be the corresponding piece-wise linear process as in \Cref{construct:BB} for some $x \in [0, 1]^{n + 1}$. Moreover, let $\xi = \mathcal{H}(U_1), \; W^D = T^{(n)}(0, U_{2:n+1})$ for a uniformly distributed random variable $U \sim \lambda^n|_{[0, 1]^{n+ 1}}$ in $[0, 1]^{n + 1}$. Then we have the weak convergence
    $$
    L^P(\bm \xi, \bm W^D) \overset{d}{\underset{P \to \infty}{\longrightarrow}} \mathcal{L}(\xi) \otimes \mathcal{L}(W^D)
    $$
\end{lemma}

\begin{proof}
    We prove the result first for $x = 0$. The case of general $x \in [0, 1]^{n + 1}$ follows then by a slight modification. Let $g: \R \times C_T \to \R$ be bounded and continuous. We define $f:[0, 1]^{n + 1} \to \R \times C_T$ as
    $$
    f(x) = \begin{cases} (\mathcal{H}(x_1), (T^{(n)}(x_{2:n + 1}, 0)), & x \in (0, 1)^{n+1}, \\ (0, \bm 0), & \; \text{else}, \end{cases}
    $$
    where $\bm 0$ denotes the constant zero function in $C_T$. Let us first establish the continuity of $f$ in $(0, 1)^{n+1}$. This is equivalent with continuity of $f|_{(0, 1)^{n + 1}}$. The supremum metric
    \[
    d((z_1, f_1), (z_2, f_2)) = \max\{ |z_1 - z_2|, d_{\infty}(f_1, f_2)\}.
    \]
    metrizes the product topology on $\R \times C_T$. Take any $(y_1, x_1), (y_2, x_2) \in (0, 1) \times (0, 1)^{n}$. Since the maximal distance between two piece-wise linear functions on the same dissection is attained at one of the nodes, we find that
    \begin{align*}
        d(f(y_1, x_1), f(y_2, x_2)) = \max\bigg\{  |\mathcal{H}(y_1)& - \mathcal{H}(y_2)|, \\
        & \left. \max_{0 \leq i \leq \#D - 1}\left|T^{(n)}(x_1, 0)(t_i) - T^{(n)}(x_2, 0)(t_i) \right|\right\}.
    \end{align*}
    Hence, continuity of $f|_{(0, 1)^n}$ breaks down to continuity of 
    \begin{align*}
        x \in (0, 1)^{n + 1} \longmapsto (\mathcal{H}(x_1), T^{(n)}(x_{2:}, 0)(t_0), \hdots, T^{(n)}(x_{2:}, 0)(t_{\#D - 1})) \in \R^{\#D}.
    \end{align*}
    But this map is clearly continuous; $\mathcal{H}$ is by assumption continuous and in the other components, it is component-wise a polynomial of degree 1 in the Gaussian inverse distribution function $\phi^{-1}$, which is known to be continuous. Thus,
    \begin{align*}
        g\circ f : [0, 1]^{n + 1} \longrightarrow \R
    \end{align*}
    is almost surely continuous and bounded, hence Riemann-integrable. Overall,
    \begin{align*}
        \int g dL^P(\bm \xi, \bm W^D) &= \frac{1}{P} \sum_{i = 1}^P g(\xi_i, W_i^D) \\
        & = \frac{1}{P} \sum_{i = 1}^P g \circ f(\zeta^i_{1:n+1}) \\
        & \underset{P \to \infty}{\longrightarrow} \E[g \circ f(U)] \\
        & = \E[g(\xi, W^D)],
    \end{align*}
    which proves the assertion. In case $x \neq 0$ we exclude the now modified null set of discontinuities of $f$ and proceed as before.
\end{proof}

The last piece in the puzzle is the weak convergence of the piece-wise linear approximation.

\begin{lemma}\label{lem:Wong-Zakai}
    For a Wiener process $W$ and its piecewise linear approximation $W^D$ we have
    \begin{align*}
        \mathcal{L}(W^D) \overset{d}{\longrightarrow} \mathcal{L}(W).
    \end{align*}
\end{lemma}

\begin{proof}
    As a straightforward consequence of the fact that $W$ is pathwise uniformly continuous on $[0, T]$ we obtain almost surely pathwise the convergence
    \begin{align*}
        W^D(\omega) \underset{|D| \to 0}{\longrightarrow} W(\omega)
    \end{align*}
    in the uniform topology on $C_T$. Now, let $g:C_T \to \R$ be a bounded continuous function. Then the above and dominated convergence imply
    \begin{align*}
        \E[g(W^D)] \underset{|D| \to 0}{\longrightarrow} \E[g(W)],
    \end{align*}
    which proves the claim.
\end{proof}

The proof of the main result follows.

\begin{proof}[Proof of \Cref{thm:weak_convergence}]
    Let $\mathcal{P}(\R \times C_T)$ be the space of probability measures on $\R \times C_T$ and define respectively $\mathcal{P}(C_T)$. We define the solution map
    \begin{align*}
        \Phi: \mathcal{P}(\R \times C_T) \to \mathcal{P}(C_T), \; \mathcal{L}(\xi, W) \longmapsto \mathcal{L}(Z),
    \end{align*}
    which matches the initial distribution and noise with the respective solution of \Cref{additive_noise_MVSDE}. We denote the law of the corresponding process at time $t$ by $\Phi_t(\cdot)$. Moreover, let $D_T$ denote the space of càdlàg paths on $[0, T]$, equipped with the Skorokhod metric. In \cite{friz2020}, it was shown that the pendant of $\Phi$ on $\mathcal{P}(\R \times D_T)$ is continuous in the weak topology. This result implies that $\Phi$ is so as well. To see this, notice that $C_T$ is a closed subspace of $D_T$ and that the induced subspace topology coincides with the uniform topology, as the respective norms have in $C_T$ the same convergent sequences. Hence, one can use Tietze-Urysohn to extend continuous functions from $C_T$ to $D_T$ to conclude the weak continuity of $\Phi$. Let now $\delta > 0$. Then \Cref{lem:global_error} and \Cref{lem:Wong-Zakai} yield a $\delta> 0$ such that for a dissection $D$ with $|D| < \delta$, a uniform random variable $U$ on $[0, 1]^{\#D}$, $\xi_i(U) = \mathcal{H}(\zeta^i_1 + U_1)$, $W_i^D(U) = T^{(\#D)}(\zeta^i_{2:n+1}, U_{2:n + 1})$,
    \begin{align*}
        \left| \E\left[ \frac{1}{P} \sum_{i = 1}^P (g(X_i^E(T)) \right] - \right. & \left. \E\left[\int g \; d\Phi_T(L(\bm \xi(U), \bm W^D(U)))\right] \right| \\
        & = \left|\E\left[ \frac{1}{P} \sum_{i = 1}^P \left(g(X_i^E(T)) - g(X_i^D(T))\right)\right] \right| \\
        & < \frac{\varepsilon}{3},
    \end{align*}
    and
    \begin{align*}
        \bigg|  \E\left[ \int g \; d\Phi_T(\mathcal{L}(\xi, W^D)) \right] & - \E[g(Z(T))] \bigg| \\ 
        & = \left|  \E\left[ \int g \; d\Phi_T(\mathcal{L}(\xi, W^D)) \right] - \int g \; d\Phi_T(\mathcal{L}(\xi, W)) \right| \\
        & < \frac{\varepsilon}{3}.
    \end{align*}
    Both estimates hold for any number $P$ of particles. For such a $D$ with $|D| < \delta$, a simple application of dominated convergence and \Cref{lem:weak-P-convergence} yields then a $\tilde P \in \N$ such that for any $P > \tilde P$,
    \begin{align*}
        \left|  \int g \; d\Phi_T(\mathcal{L}(\xi, W^D)) - \E\left[\int g \; d\Phi_T(L(\bm \xi(U), \bm W^D(U)))\right] \right| < \frac{\varepsilon}{3}.
    \end{align*}
    The triangle inequality finally yields for this $D$ and all $P > \tilde P$,
    \begin{align*}
        \bigg| \E\bigg[ g(Z(T))& - \frac{1}{P} \sum_{i = 1}^P g(X^E_i(T))\bigg] \bigg| \\
        \leq & \;  \left| \E\left[ \frac{1}{P} \sum_{i = 1}^P (g(X_i^E(T)) \right] - \E\left[\int g \; d\Phi_T(L(\bm \xi(U), \bm W^D(U)))\right] \right| \\
        & \; + \left|  \int g \; d\Phi_T(\mathcal{L}(\xi, W^D)) - \E\left[\int g \; d\Phi_T(L(\bm \xi(U), \bm W^D(U)))\right] \right| \\
        & \; + \left|  \int g \; d\Phi_T(\mathcal{L}(\xi, W^D)) - \E\left[g(Z(T)\right] \right| \\
        < & \; \varepsilon.
    \end{align*}
\end{proof}

\subsection{Rates in a Simplified Setting}\label{subsec:rates}
In this section, we prove weak and strong convergence rates. However, although our method works numerically very well for McKean--Vlasov SDEs, the circumstances -- high dimensionality and unbounded variation of the Gaussian inverse transform, to name a few -- force us to move to a simpler setting where these problems do not occur and the famous inequality of Koksma and Hlawka can be applied. That being said, we concentrate on zero noise McKean--Vlasov differential equations of type

\begin{align}\label{eq:MFO}
Z(t) = \xi + \int_{0}^t a\left(Z(s), \int \kappa(Z(s), z) \mathcal{L}(Z(s))(dz) \right) ds, \quad t \in [0, T],
\end{align}
which are ODEs with random initial conditions. In the proceeding, this is often referred to as the Mean Field ODE. Our goal is to show that under certain regularity assumptions on drift, initial condition, and kernel, the continuous time limit of our QMC Particle system estimator approximates the solution to \Cref{eq:MFO} and that the error is bounded by the square of the star discrepancy of the input points $\zeta^1, \hdots, \zeta^P$,
\[
\Delta_{P, N}^* := \sup_{x \in [0, 1]^N} \left|\frac{1}{P} \sum_{i = 1}^P 1_{[0, x]}(\zeta^i) - \lambda ([0, x]) \right|,
\]
multiplied by at most a constant. For convenience, we will use below the notation $\Delta_{\zeta}$ for this very same object. Furthermore, if the sequence is uniformly shifted by $x \in [0, 1]$ and mapped back to $[0, 1]$ by taking the fractional part, the discrepancy of the resulting sequence will be compactly denoted by $\Delta_{\{\zeta + x\}}$. Note that in this section, only $1$-dimensional points are needed. The discretization does not introduce further dimensions, as there is no noise, only mere discretization in time. The latter can be handled with standard ODE techniques, which is why we are allowed to work directly with the associated continuous-time particle system.

\subsubsection{Preparations}\label{subsec:prep}
We start defining suitable regularity assumptions and prepare the setting with existing results from the literature for our subsequent needs.

For the subsequent analysis, we must have at least a differentiable mapping from QMC points to the solution of the occurring equations. The first obstacle is thus the modulo operation used to send shifted points back to $[0, 1]$. To smooth out this operation, we restrict our initial conditions to a certain class of probability laws. $\R$ will in the following always be equipped with the Borel $\sigma$-field, and $\lambda$ denotes the Lebesgue measure. Let $\mu$ be a probability measure on $\R$. We say that $\mu$ is $C^k$-periodizable if it can be written as $\mu = 1_{[0, 1]}\lambda \circ f^{-1}$ for a periodic function $f \in C^k(\R)$ with period $1$. We fix a setting that is assumed to hold throughout this subsection.

\begin{assumption}\label{ass:setting}
    For this entire section, we will always stay within the following setting. Let
    \begin{enumerate}
        \item[i)] $a: \R \times \R \to \R$ be a bounded function that has continuous partial derivatives up to the second order; moreover, let these derivatives be bounded.
        \item[ii)] $\xi$ be a random variable such that $\mathcal{L}(\xi)$ is $C^1$-periodizable,
        \item[iii)] $\kappa : \R \times \R \to \R$ be a bounded function that has continuous and bounded first-order partial derivatives
    \end{enumerate}
\end{assumption}

It is evident from \cite{hajiali2021simple} that the assumptions above are sufficient to guarantee the existence and uniqueness of solutions to \Cref{eq:MFO} as well as to the particle system

\begin{align}\label{ode-ps}
X_i^P(t) = \xi_i + \int_0^t a\left(X_i^P(s), \frac{1}{P} \sum_{j = 1}^P \kappa(X_i^P(s), X_j^P(s)) \right) ds, \; 1 \leq i \leq P, \; t \in [0, T]
\end{align}

in the mean square sense, where $\xi_i$ are identically distributed copies of $\xi$, not necessarily independent. Further, they are sufficient to guarantee the existence and uniqueness of pointwise classic ODE solutions to \Cref{ode-ps}, which of course coincides with the mean square solution. We refer the reader to \cite{HartmanODE}. Further, it boils down to a simple calculation involving chain rule, Fubini's theorem, and the previous assumptions to check that the map
\[
(x, t) \mapsto E[\kappa(x, Z(t))]
\]
is $C^1$ and bounded on $\R \times [0, T]$. This is crucial for our analysis since it implies that $y(t; \xi)$ is a representation of the solution to \Cref{eq:MFO}, where $y$ is the pointwise solution of the ODE

\begin{align*}
        y(t; y_0) = y_0 + \int_{0}^t a\left(y(s; y_0), \int \kappa(y(s; y_0), z) \mathcal{L}(Z(s))(dz) \right) ds.
\end{align*}

In particular, this equation admits a continuously differentiable flow $\Phi^t(y_0) = y(t; y_0)$, that is continuous in time. All these results are classic and we refer the reader to \cite{HartmanODE} for reference. Let now $f: \R \to \R$ be the associated periodic $C^1$ map to the initial condition $\xi$. We state a technical proposition, required for later error bounds. These involve the variation $V(g)$, which in case $g$ is differentiable satisfies
\[
V(g) = \int_{[0, 1]} |g'| dx,
\]
and the desired mapping of points to solutions $\Psi^t := \Phi^t \circ f \circ \{\cdot\}$. It is of course superfluous to take the fractional part when passing an argument to $f$, due to periodicity. We do this nevertheless to remind ourselves and the reader that we perform integration on $[0, 1]$.

\begin{proposition}[Bounded Variations] \label{lem:VHK}
    Defining $\Psi^t := \Phi^t \circ f \circ \{\cdot\}$, we find that
    \begin{align*}
        \sup_{x, t} V(\kappa(\Psi^t(x), \Psi^t(\cdot))) & < \infty \\
        \sup_{x, t} V(\kappa_x(\Psi^t(x), \Psi^t(\cdot))) & < \infty.
    \end{align*}
\end{proposition}

\begin{proof}
    We observe that $(x, t) \to \Psi^t(x)$ is continuous in $t$ and continuously differentiable in $x$. Since $[0, T] \times [0, 1]$ is compact, the result follows immediately from the characterization of the variation through the derivative.
\end{proof}

We proceed with a short discussion of the properties of the particle system. Having a twice continuously differentiable function $g: \R^P \to \R$ with bounded derivatives, a probabilistic interpretation of the particle system yields the flow 
\[
u(x_1, \hdots, x_P, t) = \E[g(X_1^P(T), \hdots, X_P^P(T)) \mid X_i^P(t) = x_i, \; 1 \leq i \leq P ]
\]
satisfying the partial differential equation
\begin{align}\label{eq:kbe}
\begin{split}
    u_t &= \sum_{i = 1}^P a\left( x_i, \frac{1}{d} \sum_{j = 1}^P \kappa(x_i, x_j) \right) u_{x_i}, \quad t \in [0, T], \; x \in \R^P, \\
    u(T, x) &= g(x).
\end{split}
\end{align}
Since the system in \Cref{ode-ps} is essentially a family of ODEs, one may as well consult \cite{HartmanODE}, Theorem 3.1 to arrive at the same conclusion without the probabilistic interpretation. We gratefully accept the flow bounds for particle systems proved in \cite{hajiali2021simple}, implying for a family $g_P : \R^P \to \R$ of bounded and symmetric functions with continuous partial derivatives up to order $2$, such that there exists a $0 < C < \infty$, so that for all $P \in \N$,
\begin{align}\label{eq:bound_g}
    \sum_{0 \leq |k| \leq 2, k \in \N^P} \norm{\frac{\partial^{|k|}}{\partial x^k} g_P }_{L^\infty} \leq C,
\end{align}
the corresponding flows $u$ satisfy, uniformly in $t \in [0, T]$,
\begin{align}\label{eq:bound_u}
\begin{split}
    & \sum_{i = 1}^P \sup_{x \in \R^P} \left| u_{x_i}(t, x) \right| \leq C < \infty, \\
    & \sum_{j = 1}^P \sum_{i = 1}^P \sup_{x \in \R^P} \left| u_{x_i x_j}(t, x) \right| \leq C < \infty,
\end{split}
\end{align}
where the constant $C$ does not depend on $P$. In particular, if the $g_P$ are additionally symmetric, we obtain, again uniformly in $t \in [0, T]$,
\begin{align}\label{eq:bound_sym}
\begin{split}
    & \sup_{x \in \R^P}  \left| u_{x_i}(t, x) \right| \lesssim \frac{1}{P}, \; 1 \leq i \leq P,\\
    &  \sum_{j = 1}^P  \sup_{x \in \R^P} \left| u_{x_i x_j}(t, x) \right| \lesssim  \frac{1}{P},  \; 1 \leq i \leq P
\end{split}
\end{align}
where we write '$\lesssim$' if the term on the right-hand side is a bound after multiplication by a constant and of course, the constant in \Cref{eq:bound_sym} does again not depend on $P$. We note further that we use the notation $\norm{\cdot}_{L_{\infty}}$ for the essential supremum norm.

\subsubsection{Mean Field Convergence}
We start with the strong error bound. The proof follows a Talay-Tubaro argument involving the flow of the particle system. A similar technique was applied in \cite{hajiali2021simple}.

\begin{theorem}\label{thm:strong_error}
Let $X^P = \{X_i^P\}_{i= 1}^P$ be the solution to \Cref{ode-ps} and let $Z = \{Z_i\}_{i = 1}^P$ be solutions to \Cref{eq:MFO}, both with initial conditions $f(\{\zeta^i + U\})$ for a uniform random variable $U$, a periodic function $f \in C^1(\R)$ with period $1$ and a sequence of points $\zeta^1, \hdots, \zeta^P \in [0, 1]$. Moreover, let $g_P : \R^P \to \R$ be a family of bounded functions as introduced in \Cref{subsec:prep}. Then the strong error satisfies
    \begin{align*}
        \E\left[\left(g_P(X^P(T)) - g_P(Z(T))\right)^2\right] \lesssim \sup_{x \in [0,1]} \Delta_{\{\zeta + x\}}^2,
    \end{align*}
    where the constant does not depend on $P$.
\end{theorem}
Note that contrary to the weak error, we assume no symmetry on the functions $g_P$.
\begin{proof}
Defining $U(t) := u(t, Z(t))$, an application of the fundamental Theorem of calculus, the chain rule, \Cref{eq:MFO} and \Cref{eq:kbe} yields
\begin{align*}
g_P(Z(T)) - g_P(X^P(T)) & = U(T) - U(0) \\
& = \int_0^T \frac{d}{dt} u(t, Z(t))dt \\
& =  \int_0^T u_t(t, Z(t)) + \sum_{i = 1}^P u_{x_i}(t, Z(t)) \frac{d}{dt} Z(t) dt \\
& = \int_0^T \sum_{i = 1}^P \Delta_i a \ddi{u}{x_i} dt,
\end{align*}
where
\begin{align*}
    \Delta_i a = a\left(Z_i, \int \kappa(Z_i(t), z) \mathcal{L}(Z(t))(dz) \right) - a\left(Z_i, \frac{1}{P} \sum_{i = 1}^P \kappa(Z_i(t), z) \right).
\end{align*}
Using a first-order Taylor expansion, \Cref{lem:VHK}, and \Cref{eq:bound_u} we bound pointwise,
\begin{align*}
    \left|\sum_{i = 1}^P \Delta_i a \ddi{u}{x_i} \right| & \lesssim \norm{\frac{\partial a}{\partial y}}_{L^\infty} \vast| \int \kappa\left(Z_i(t), z\right) \mathcal{L}(Z(t))(dz) - \frac{1}{P} \sum_{j = 1}^P \kappa(Z_i(t), Z_j(t)) \vast| \\
    & = \norm{\frac{\partial a}{\partial y}}_{L^\infty} \vast| \int \kappa\left(\Psi^t(\{\zeta^i + U \}), z\right) \mathcal{L}(Z(t))(dz)  \\
    & \hspace{2.7cm} - \frac{1}{P} \sum_{j = 1}^P \kappa(\Psi^t(\{\zeta^i + U \}) , \Psi^t(\{\zeta^j + U \})) \vast| \\
    & \lesssim \sup_{x \in [0, 1]}  \Delta_{\{\zeta + \tilde x\}} V_{HK}(\kappa(\Psi^t(x), \Psi^t(\cdot))) \\
    & \lesssim \sup_{x \in [0, 1]}  \Delta_{\{\zeta + x\}},
\end{align*}
with the constant not depending on $P$ or $t$. This gives, applying the bound pointwise,
\[
    \E\left[\left(g_P(X^P(T)) - g_P(Z(T))\right)^2 \right] \lesssim \Delta_{\zeta + x}^2,
\]
which completes the proof.
\end{proof}

We continue with a proof of the weak error rate, which is more subtle. The proof relies on the group structure of sequences such as the rank-1 lattice points which restores the indistinguishability of particles in our construction. We state an almost obvious proposition addressing this property which is essential for the proof.

\begin{proposition}\label{prop:perm_map}
    Let $\zeta^1, \hdots, \zeta^P \in [0, 1]$ be a sequence of points that form a finite subgroup of $\R / \Z$. Let $g: [0, 1]^{P} \to \R$ be a measurable function that is invariant under permutations $\pi$ of $\{1, \hdots, p\}$ that fix the first element, i.e. $\pi(1) = 1$. Then 
    \begin{align*}
        f: [0, 1] &\to \R, \\
        & x \mapsto g(\{ \zeta^1 + x\}, \hdots, \{ \zeta^{P} + x\}) \\
    \end{align*}
    satisfies for any permutation in the symmetric group $\pi \in S_P$, and $s \in [0, 1]$ that
    $$f(\{ \zeta^{\pi(1)} - \zeta^1 + s \}) = g(\{\zeta^{\pi(1)} + s\}, \hdots, \{\zeta^{\pi(P)} + s\}).$$ In particular, given a uniform random variable $U$ on $[0, 1]$, we find $$ g(\{ \zeta^1 + U\}, \hdots, \{ \zeta^{P} + U\}) \overset{d}{=}  g(\{\zeta^{\pi(1)} + U\}, \hdots, \{\zeta^{\pi(P)} + U\}).$$
\end{proposition}

\begin{proof}
   Let $G = \{\zeta^1, \hdots, \zeta^p \}$ denote the finite group containing the points. Inherent to being a group is the property that $g + G = G$ for any $g \in G$. We further note that $\{\{a\} + \{b\}\} = \{a + b\}$ for any $a, b \in [0, 1]$. Both together already imply the first assertion. The second assertion follows then immediately from the first.
\end{proof}

We continue putting this simple fact to use to prove a weak error bound.

\begin{theorem}\label{thm:weak_error}
Let $X^P = \{X_i^P\}_{i= 1}^P$ be the solution to \Cref{ode-ps} and let $Z = \{Z_i\}_{i = 1}^P$ be solutions to \Cref{eq:MFO}, both with initial conditions $f(\{\zeta^i + U\})$ for a uniform random variable $U$, a periodic function $f \in C^1(\R)$ with period $1$ and a sequence of points $\zeta^1, \hdots, \zeta^P \in [0, 1]$ forming a finite subgroup of $\R / \Z$. Moreover, let $g_P : \R^P \to \R$ be a family of bounded functions as introduced in \Cref{subsec:prep}. Then the weak error satisfies
\[
\left| \E\left[g_P(X^P(T))- g_P(Z(T)) \right] \right| \lesssim \sup_{x \in [0, 1]} \Delta_{\{\zeta + x\}}^2,
\]
where the constant does not depend on $P$.
\end{theorem}

\begin{proof}
The strategy is somewhat similar to the one in \Cref{thm:strong_error}. Again with $U(t) := u(t, Z(t))$, an application of the fundamental Theorem of calculus, the chain rule, \Cref{eq:MFO} and \Cref{eq:kbe} we obtain
\begin{align*}
g_P(Z(T)) - g_P(X^P(T)) & = U(T) - U(0) \\
& = \int_0^T \frac{d}{dt} u(t, Z(t))dt \\
& =  \int_0^T u_t(t, Z(t)) + \sum_{i = 1}^P u_{x_i}(t, Z(t)) \frac{d}{dt} Z(t) dt \\
& = \int_0^T \sum_{i = 1}^P \Delta_i a \ddi{u}{x_i} dt,
\end{align*}
with
\begin{align*}
    \Delta_i a = a\left(Z_i, \int \kappa(Z_i(t), z) \mathcal{L}(Z(t))(dz) \right) - a\left(Z_i, \frac{1}{P} \sum_{i = 1}^P \kappa(Z_i(t), z) \right).
\end{align*}
Next, Taylor expanding we find an $x \in \R$ that may depend on the appropriate parameters, such that
\begin{align*}
\sum_{i = 1}^P \Delta_i a \ddi{u}{x_i}  = \frac{1}{P} \sum_{i=1}^P A_i + \sum_{i = 1}^P B_i,
\end{align*}
where
\begin{align*}
    A_i = & P\frac{\partial u}{\partial x_i} \left( Z_1(t), \hdots, Z_P(t), t \right) \frac{\partial a}{\partial y}\left(Z_i(t), \int \kappa\left(Z_i(t), z\right) \mathcal{L}(Z(t))(dz) \right) \\
   & \cdot \left( \int \kappa\left(Z_i(t), z\right) \mathcal{L}(Z(t))(dz) - \frac{1}{P} \sum_{j = 1}^P \kappa(Z_i(t), Z_j(t)) \right),
\end{align*}
\begin{align*}
    B_i = &\; -\frac{\partial u}{\partial x_i} \frac{\partial^2 a}{\partial y^2}\left( Z_i(t),\int \kappa\left(Z_i(t), z\right) \mathcal{L}(Z(t))(dz) + x \right) \\
    & \;  \cdot \left( \int \kappa\left(Z_i(t), z\right) \mathcal{L}(Z(t))(dz) - \frac{1}{P} \sum_{j = 1}^P \kappa(Z_i(t), Z_j(t)) \right)^2. 
\end{align*}
Boundedness of $a, u$ and their derivatives yield
\begin{align*}
    \left| \E\left[ \sum_{i = 1}^P B_i \right] \right| \leq \sup_{x \in [0, 1]} \Delta_{\{\zeta + x\}}^2 \sup_{x \in [0, 1]} V(\kappa(\Psi^t(x), \Psi^t(\cdot))).
\end{align*}
Moreover, we define on $[0, 1]^{2}$ the map
\begin{align*}
    G(x, \tilde x) = & P \frac{\partial u}{\partial x_1} \left( \Psi^t(x), \Psi^t(x + \zeta^2 - \zeta^1, \hdots, \Psi^t(x + \zeta^P - \zeta^1), t \right) \\
    & \cdot \frac{\partial a}{\partial y}\left(\Psi^t(x), \int \kappa\left(\Psi^t(x), z\right) \mathcal{L}(Z(t))(dz) \right) \\
   & \cdot \left( \int \kappa\left(\Psi^t(x), z\right) \mathcal{L}(Z(t))(dz) - \frac{1}{P} \sum_{j = 1}^P \kappa(\Psi^t(x), \Psi^t(\tilde x + \zeta^i )) \right).
\end{align*}
In the light of \Cref{prop:perm_map}, we see immediately that $G(\zeta^i + U, U) = A_i$. Let $\tilde U$ be another random shift, independent of $U$. Independence and Fubini's theorem yield
\[
\E[G(\tilde U, U)] = \E_{\tilde U}[\E_{ U}[G(\tilde U, U)]] = 0.
\]
We conclude with a pointwise application of Koksma-Hlawka's inequality, \cite{niederreiterQMC},
\begin{align*}
\left| \E\left[ \frac{1}{P} \sum_{i = 1}^P A_i \right] \right| & = \left| \E\left[ \frac{1}{P} \sum_{i = 1}^P A_i  - G(\tilde U, U)\right] \right| \\
& = \left| \E_U\left[ \E_{\tilde U} \left[ \frac{1}{P} \sum_{i = 1}^P A_i  - G(\tilde U, U)\right] \right] \right| \\
& = \left| \E\left[  \frac{1}{P} \sum_{i = 1}^P G(\zeta^i + U, U) - \E_{\tilde U}[G(\tilde U, U)]  \right] \right| \\
& \leq \E\left[ \left| \frac{1}{P} \sum_{i = 1}^P G(\zeta^i + U, U) - \E_{\tilde U}[G(\tilde U, U)] \right| \right] \\
 & \leq \sup_{x \in [0, 1]} \Delta_{\{\zeta + x\}} \E[ V(G(\cdot, U)) ].
\end{align*}
Introducing a dummy random shift $\hat U$, we estimate pointwise, again with Koksma-Hlawka,
\begin{align*}
    \left| \frac{\partial}{\partial x} G(x, \tilde x) \right| \leq & \; \norm{\frac{d}{dx}\Psi^t}_{L^\infty} P \left( \sum_{i = 1}^P \norm{\frac{\partial^2}{\partial x_i \partial x_1} u}_{L^\infty} \right) \norm{\frac{\partial a}{\partial y}}_{L^\infty} \\
    & \; \; \cdot \left| \E[\kappa(\Psi^t(x), \Psi^t(\hat U))] - \frac{1}{P} \sum_{j = 1}^P \kappa(\Psi^t(x), \Psi^t(\tilde x + \zeta^i )) \right| \\
    & \; + \norm{P\frac{\partial u}{\partial x_1}}_{L^\infty} \norm{\frac{d}{dx}\Psi^t}_{L^\infty} \\
    & \; \; \cdot \left( \norm{ \frac{\partial^2 a}{\partial x \partial y}}_{L^\infty} + \norm{\frac{\partial^2 a}{\partial y \partial y}}_{L^\infty} \norm{\frac{d}{dx}\E[\kappa(\Psi^t(x), \Psi^t(\hat U))]}_{L^\infty} \right) \\
    & \; \; \cdot \left| \E[\kappa(\Psi^t(x), \Psi^t(\hat U))] - \frac{1}{P} \sum_{j = 1}^P \kappa(\Psi^t(x), \Psi^t(\tilde x + \zeta^i )) \right| \\
    & \; + \norm{P\frac{\partial u}{\partial x_1}}_{L^\infty} \norm{ \frac{\partial a}{\partial y} }_{L^\infty} \norm{\frac{d}{dx}\Psi^t}_{L^\infty} \\
    & \; \; \cdot \left| \E[\kappa_x(\Psi^t(x), \Psi^t(\hat U))] - \frac{1}{P} \sum_{j = 1}^P \kappa_x(\Psi^t(x), \Psi^t(\tilde x + \zeta^i )) \right| \\
    \lesssim & \; \Delta_{\{\zeta + \tilde x\}} \left(V(\kappa(\Psi^t(x), \Psi^t(\cdot))) + V(\kappa_x(\Psi^t(x), \Psi^t(\cdot))) \right),
\end{align*}
where the constant does not depend on $P$. Overall this leads to, introducing again an independent dummy shift $\hat U$,
\begin{align*}
 \E[ V(G(\cdot, U)) ] = & \E\left[\left| \frac{\partial}{\partial x} G(\hat U, U) \right| \right] \\
 \lesssim & \; \; \sup_{x \in [0, 1]} \Delta_{\{\zeta + x\}} \sup_{x \in [0, 1]} \left( V(\kappa(\Psi^t(x), \Psi^t(\cdot))) + V(\kappa_x(\Psi^t(x), \Psi^t(\cdot))) \right).
\end{align*}
We conclude,
\begin{align*}
     \E\left[ \frac{1}{P} \sum_{i = 1}^P A_i \right] & \lesssim \sup_{x \in [0, 1]} \Delta_{\{\zeta + x\}} \cdot \sup_{x \in [0, 1]} \Delta_{\{\zeta + x\}} \\
     & = \sup_{x \in [0, 1]} \Delta_{\{\zeta + x\}}^2.
\end{align*}
Taking the supremum in $t$, using compactness of $[0, T]$, together with the bound for $B_i$ this proves the result.
\end{proof}

We close our discussion with a discrepancy bound for the variance. The latter requires us to further restrict our attention to objective functions of the form $g_P(x_1, \hdots, x_P) = \frac{1}{P} \sum_{i = 1}^P g(x_i)$.

\begin{theorem}\label{thm:variance}
    Let $g:\R \to \R$ be a twice continuously differentiable function with bounded derivatives. Let further $X_i^P, i = 1, \hdots, P$ be the solution to \Cref{ode-ps} with initial conditions $f(\{\zeta^i + U\})$ for a uniform random variable $U$, a periodic function $f \in C^1(\R)$ with period $1$ and a sequence of points $\zeta^1, \hdots, \zeta^P \in [0, 1]$ forming a finite subgroup of $\R / \Z$.
    \begin{align*}
        \Var\left[ \frac{1}{P} \sum_{i = 1}^P g(X_i^P(T)) \right] \lesssim \sup_{x \in [0, 1]}  \Delta_{\{\zeta + \tilde x\}}^2,
    \end{align*}
    where the constant does not depend on $P$.
\end{theorem}

\begin{proof}
    The family $g_P(x_1, \hdots, x_P) = g(x_1)$ obviously satisfies \Cref{eq:bound_g}. Hence, the corresponding flows $u(x, t) = u_P(x, t) = \E\left[ g(X_1^P(T)) \mid X_i^P(t) = x_i, \; 1 \leq i \leq P \right]$, with a slight abuse of notation, of the particle system satisfy the flow bounds in \Cref{eq:bound_u}. We define the map $G:[0, 1] \to \R$,
    $$ G(x) = u\left(f(\{x\}), f(\{\zeta^2 + x - \zeta^1\}), \hdots, f(\{\zeta^P + x - \zeta^1\}), 0 \right). $$
    In the light of \Cref{prop:perm_map} it is clear that
    $$g(X_i^P(T)) = G(\{\zeta^i + U\}), \; 1 \leq i \leq P$$ and hence in particular,
    $$g(X_i^P(T)) \overset{d}{=} g(X_j^P(T)), \; 1 \leq i, j \leq P.$$
    We conclude, introducing an independent shift $\hat U$,
    \begin{align*}
        \Var\left[ \frac{1}{P} \sum_{i = 1}^P g(X_i^P(T)) \right] & = \E\left[ \left(\frac{1}{P} \sum_{i = 1}^P g(X_i^P(T)) - \E\left[\frac{1}{P} \sum_{i = 1}^P g(X_i^P(T))\right] \right)^2 \right] \\
        & = \E\left[ \left(\frac{1}{P}\sum_{i = 1}^P G(\{\zeta^i + U\}) - \E\left[G(\hat U)\right] \right)^2 \right] \\
        & \leq \E\left[ \left( \Delta_{\{\zeta + U\}} V(G) \right)^2 \right] \\
        & \leq \sup_{x \in [0, 1]} \Delta_{\{\zeta + x\}}^2 V^2(G).
    \end{align*}
    Due to the periodicity of $f$, we are allowed to carelessly differentiate $G$ and arrive at
    \begin{align*}
        \bigg|\frac{d}{dx} & G(x)\bigg| \\
        & = \left|\sum_{i = 1}^P   \frac{\partial}{\partial x_i} u\left(f(\{x\}), f(\{\zeta^2 + x - \zeta^1\}), \hdots, f(\{\zeta^P + x - \zeta^1\}), 0 \right) f'(\{ \zeta^i + x - \zeta^1\})\right| \\
        & \leq \norm{f'}_{L^{\infty}} \sum_{i = 1}^P \norm{\frac{\partial}{\partial x_i} u(\cdot, 0)}_{L^{\infty}} \\
        & \leq C,
    \end{align*}
    where the constant does not depend on $P$ due to the flow bounds. To convince ourselves that $\norm{f'}_{L^{\infty}} < \infty$, we note that $f'$ is by assumption continuous and inherits periodicity from $f$. This yields
    $$ V(G) \leq C,$$
    and thus the result.
\end{proof}

\begin{remark}
    The boundedness and symmetry condition on the family of functions $g_n$ deserves some justification. Let $g:\R \to \R$ be a bounded, twice continuously differentiable function with bounded derivatives. Then the family
    \begin{align*}
        g_n(x_1, \hdots, x_n) := \frac{1}{n} \sum_{i = 1}^n g(x_i)
    \end{align*}
    satisfies both assumptions. This shall suffice as motivation, as it covers an error bound when estimating the quantity $\E[g(Z)]$, which is our final goal when approximating a McKean--Vlasov dynamic.
\end{remark}

\begin{remark}
    Certainly, methods of even higher order can be constructed for mean-field ODEs. However, the focus of this work is the remarkably fast convergence for McKean-Vlasov SDEs, where this is not easily possible and this section aims to provide a first intuition on where this convergence rate comes from.
\end{remark}

\section{Antithetic MLQMC}\label{sec:multilevel}
\FloatBarrier

In this section, we present a multilevel estimator while making use of the properties of the rank 1 lattice rule, which enables a telescopic multilevel construction. Generally, splitting a sequence of non-equidistant points into two parts followed up by randomization with one shift for all points does not yield equally distributed random variables. This is indeed true for certain partitions of sequences generated by a rank 1 lattice rule.

\begin{lemma} \label{lem:even_split}
Let $\bm{\zeta} = \left\{ \zeta^k = \left\{ \frac{kz}{2^\ell} \right\} \mid 0 \leq k \leq 2^{\ell} - 1 \right\}$ be a sequence from a rank 1 lattice rule with some generating vector $z \in \R^d$ and $\ell \geq 2$. Let $U \sim U([0, 1]^d)$ be a random shift. Then, the random $d \times 2^{\ell - 1}$ vectors
\begin{align*}
\bm{\zeta}^1 & = \left[ \left\{ \zeta^0 + U \right\}, \left\{ \zeta^2 + U \right\}, \hdots, \left\{ \zeta^{2^{\ell} - 2} + U\right\} \right] \\
\bm{\zeta}^2 & = \left[ \left\{ \zeta^1 + U \right\}, \left\{ \zeta^3 + U \right\}, \hdots, \left\{ \zeta^{2^{\ell} - 1} + U \right\} \right] \\
\end{align*}
follow the same distribution.
\end{lemma}

\begin{proof}
It is sufficient to find a shift $U \in [0, 1]^d$ such that
\[
\left[ \left\{ \zeta^0 + U \right\},  \left\{ \zeta^2 + U \right\}, \hdots, \left\{ \zeta^{2^{\ell} - 2} + U \right\} \right] = \left[  \zeta^1 , \zeta^3 , \hdots,  \zeta^{2^{\ell} - 1}  \right]
\]
Choosing $U := \left\{ \frac{z}{2^{\ell}} \right\}$ and noticing that for any numbers $a, b \in \R^d$ we have
\[
\{ \{a\} + \{b\} \} = \{a + b\}
\]
a simple computation reveals that
\begin{align*}
\left[ \left\{ \zeta^0 + U \right\},  \left\{ \zeta^2 + U \right\}, \hdots, \right. & \left. \left\{ \zeta^{2^{\ell} - 2} + U \right\} \right] \\
& = \left[ \left\{  \frac{z}{2^{\ell}} \right\},  \left\{ \left\{ \frac{2z}{2^{\ell}} \right\} + \left\{ \frac{z}{2^{\ell}} \right\} \right\}, \hdots, \left\{ \left\{ \frac{(2^\ell - 2)z}{2^{\ell}} \right\} + \left\{ \frac{z}{2^{\ell}} \right\} \right\} \right] \\
& = \left[ \left\{  \frac{z}{2^{\ell}} \right\},  \left\{ \frac{3z}{2^{\ell}}  \right\}, \hdots, \left\{ \frac{(2^\ell - 1)z}{2^{\ell}} \right\} \right] \\
& = \left[  \zeta^1 , \zeta^3 , \hdots,  \zeta^{2^{\ell} - 1}  \right].
\end{align*}
\end{proof}

More splittings would have been possible in the previous lemma. Our choice has a reason that becomes clear in the following.

\begin{remark} \label{rem:split_qmc}
Let $\bm{\zeta}, \bm{\zeta}^1, \bm{\zeta}^2$ be as in \Cref{lem:even_split}. The reader may notice that
\[
\left[ \left\{ \zeta^0 + U \right\}, \left\{ \zeta^2 + U \right\}, \hdots, \left\{ \zeta^{2^{\ell} - 2} + U\right\} \right]  = \left[ \left\{  \frac{0 \cdot z}{2^{\ell - 1}} \right\},  \left\{ \frac{z}{2^{\ell - 1}}  \right\}, \hdots, \left\{ \frac{(2^{\ell - 1} - 1)z}{2^{\ell - 1}} \right\} \right],
\]
which is exactly the rank 1 lattice rule sequence for generating half as many points with the same generating vector. This makes it the perfect choice to pass it to the next lower level in a hierarchical formulation.
\end{remark}

We note that we can now construct a multilevel estimator maintaining a telescopic sum in expectation. An application of extrapolation to MLMC was previously considered in the literature \cite{Lemaire_2017}. We combine Richardson extrapolation in time with our QMC construction to set the convergence speed of the two hierarchies on equal ground. However, in doing this, we need to be even more careful with cutting out points. Within one level we need now three different sizes of QMC points. The $\Phi^{\ell}$ part needs a $P_\ell$ particle system with discretization parameter $2N_\ell$ an one with $N_\ell$. The $\Psi^\ell$ part needs one $P_{\ell - 1}$ particle system with discretization parameter $N_\ell$ and one $P_{\ell - 1}$ particle system with discretization parameter $\frac{1}{2} N_{\ell} = N_{\ell - 1}$. We will give the concrete estimator in the following.

\begin{notation}
Let $x = (x_1, \hdots, x_{N + 1}) \in \R^{N + 1}$ and $N \in \N$ with $N \mod 2 = 0$. We write
\[
x_{:2} := (x_1, x_3, \hdots, x_{N + 1}) \in \R^{\frac{N}{2} + 1}.
\]
Let $\bm{\zeta} = \{ \zeta^0, \hdots, \zeta^{P - 1} \}$ be a set of $P$ points from the rank 1 lattice rule, where $P = 2^{\ell}$ for some $\ell > 0$. We write
\begin{align*}
\bm{\zeta}_{:2} & = \{\zeta^0_{:2}, \hdots, \zeta^{P - 1}_{:2} \} \\
\bm{\zeta}_{:2}^0 & = \{\zeta^0_{:2}, \zeta^2_{:2}, \hdots, \zeta^{P - 2}_{:2} \}, \\
\bm{\zeta}_{:2}^1 & = \{\zeta^1_{:2}, \zeta^3_{:2}, \hdots, \zeta^{P - 1}_{:2} \}.
\end{align*}
\end{notation}

The notation above prepares us for defining a multilevel estimator level-wise. Due to \Cref{rem:split_qmc} we can use a new QMC sequence generated by the rank 1 lattice rule with one restriction. The generating vector $z$ needs to be the one of the top level with the right components cut out.

\begin{construction}[Antithetic Multilevel Estimator]
Let $L \in \N$ be the number of levels. For $0 \leq \ell \leq L$ let $M_\ell$ be the number of samples for each level. We fix a hierarchy of discretization parameters $N_\ell = 2^{n_0} \cdot 2^{\ell}$ and a hierarchy of particle numbers $P_\ell = 2^{p_0} \cdot 2^\ell$. We assume $n_0, p_0 > 0$. As before, the $N_\ell, P_\ell$ must be powers of $2$ at each level to ensure that the Brownian Bridge mapping of the QMC points works. We fix a generating vector $z \in \R^{2N + 1}$. For each $\ell$ let $\bm{\zeta}^\ell = \{\zeta^{0, \ell}, \hdots, \zeta^{P_\ell, \ell} \} \subset [0, 1]^{2 N_\ell + 1}$ a set of QMC points generated by a rank 1 lattice rule with generating vector
$$\Lambda^{L - \ell} z,$$
where
$$\Lambda z =  z_{:2},$$
created from $z$ by applying $L - \ell$ cuts as described above. Let $U^\ell$ be a random shift of appropriate size. We assume that we are given a function $g: \R \to \R$. We introduce the notation
\begin{align*}
\prescript{Ext}{}{\Phi}^\ell (\bm{\zeta}^\ell, U^\ell) := & \; \frac{2}{P_\ell} \sum_{i = 1}^{P_\ell} g\left( X_i^{P_\ell, N_{\ell + 1}} (\bm{\zeta}^\ell, U^\ell) \right) - \frac{1}{P} \sum_{i = 1}^{P_\ell} g\left( X_i^{P_\ell, N_\ell}(\bm{\zeta}_{:2}^\ell, U_{:2}^\ell) \right), \\
\prescript{Ext}{}{\Psi}^\ell (\bm{\zeta}^\ell, U^\ell) := & \; \frac{2}{P_\ell} \sum_{i = 1}^{P_{\ell - 1}} \left( g\left( X_i^{P_{\ell - 1}, N_\ell}((\bm{\zeta}^\ell )_{:2}^0, U^\ell_{:2}) \right) + g\left( X_i^{P_{\ell - 1}, N_\ell}((\bm{\zeta}^\ell )_{:2}^1, U^\ell_{:2}) \right) \right) \\
- \frac{1}{P_\ell} \sum_{i = 1}^{P_{\ell - 1}} & \left( g\left( X_i^{P_{\ell - 1}, N_{\ell - 1}}( (\bm{\zeta}_{:2}^\ell)_{:2}^0, (U_{:2}^\ell)_{:2}) \right) + g\left( X_i^{P_{\ell - 1}, N_{\ell - 1}}( (\bm{\zeta}_{:2}^\ell)_{:2}^1, (U_{:2}^\ell)_{:2}) \right) \right),
\end{align*}
where $X_i^{P_\ell, N_{\ell}}$ denote the respective simple Euler-Maruyama approximations.
The multilevel estimator is then given by
\begin{align*}
\mathcal{A}_{MLMC} = & \; \frac{1}{M_0} \sum_{j = 1}^{M_0}\prescript{Ext}{}{\Phi}^0(\bm{\zeta^0}, U^{(0)}_j) \\
& \; + \sum_{\ell = 1}^L  \frac{1}{M_\ell} \sum_{j = 1}^{M_\ell} \left( \prescript{Ext}{}{\Phi}^\ell(\bm{\zeta^\ell}, U^{(\ell)}_j) - \prescript{Ext}{}{\Psi}^\ell(\bm{\zeta^\ell}, U^{(\ell)}_j) \right).
\end{align*}
\end{construction}

\section{Asymptotic Error Analysis}\label{sec:errana}

\subsection{Single Level}

In this chapter, we shall analyze the numerical error of the single-level estimator as well as give an estimate of the expected total work when applying the estimator defined in \Cref{def:single-level-estimator}. For the rest of the section, we assume that $g: \R \to \R$ is a smooth function with bounded derivatives. Let $\mathcal{A}(M, P, N)$ be as in \Cref{def:single-level-estimator}. We start with some useful notation. 

\begin{notation}
Given $g$ as above, $x \in \R^P$, we abbreviate
\[
G^P(x) := \frac{1}{P} \sum_{j = 1}^P g(x_i).
\]
Given realizations $X_{p, N}^{P}, \; 1 \leq p \leq P$ of the particles at the final time, we write $\bm{X}_N^P$ for the vector whose components are given by these realizations. We denote the continuous time limit of the system in \Cref{def:single-level-estimator} by
\[
\bm{X}^P(T) := \lim_{N \to \infty} \bm X_N^P.
\]
\end{notation}

We continue with some conjectures about the convergence regarding particle system size and step size. It is well known that the Euler-Maruyama scheme converges at a weak rate $\Ocal(N^{-1})$ given that drift, diffusion, and objective function are sufficiently smooth, \cite{KOHATSUHIGA2017138}. Since the driving martingale changes here to a coupled Brownian motion, it is not completely clear that this still holds but the numerics indeed suggest it. Thus, we put it here as an conjecture.

\begin{conjecture}[Weak discretization error]\label{ass:euler}
Let $\bm X_{N}^P$ be a realization of the particle system in \Cref{def:single-level-estimator}. We assume that the weak error w.r.t. the continuous time limit is of order $\Ocal(N^{-1})$, i.e.
\[
\left|\E\el G^P\kl \bm X_{N}^P \kr - G^P\kl \bm X(T) \kr \er \right| \lesssim N^{-1}.
\]
\end{conjecture}

In the literature, it was shown that the non-QMC particle system approximation converges at a weak rate $\Ocal(P^{-1})$, \cite{hajiali2021simple}. As can be observed in the numerical experiments, tests on a linear as well as on a non-linear model with exact solutions and accurate reference solutions suggest that our approach converges twice as fast. This is also underpinned by our weak error bound in the zero diffusion context.

\begin{conjecture}[Weak particle system convergence]\label{ass:weak_conv}
Let $\bm X^P$ be the continuous-time particle system approximation and let $Z(t)$ denote the true solution to the corresponding McKean--Vlasov equation. We assume that the weak error behaves as $\Ocal(P^{-2})$, i.e.
\[
\left| \E\el G^P\kl \bm X^P(T) \kr - g(Z(T)) \er \right| \lesssim P^{-2}.
\]
\end{conjecture}

The same observation as described above was also observed for the strong rate of convergence, jumping from an initial $\Ocal(P^{-\frac{1}{2}})$ to $\Ocal(P^{-1})$ which leads to the following conjecture about the variance decay, again underpinned by the variance bound in a simplified case.

\begin{conjecture}[Variance decay]\label{ass:var}
Let $\mathcal{A}(M, P, N)$ be given as in \Cref{def:single-level-estimator}. Given the number of samples $M$ and the number of particles $P$ we assume that the variance is of order $\Ocal(M^{-1}P^{-2})$, i.e.
\[
\var\el G^P\kl \bm X_N^P \kr \er \lesssim P^{-2}.
\]
\end{conjecture}

We continue splitting particle system, discretization, and statistical error. According to the variance conjecture the overall variance of the estimator is given by
\begin{align*}
\var\el \mathcal{A}(M, P, N) \er &= \var \el \frac{1}{M} \sum_{i = 1}^M G(\bm X_N^P) \er \\
& = \frac{1}{M} \var\el G(\bm X_N^P) \er \\
& \lesssim M^{-1} P^{-2}.
\end{align*}
where the second equality uses that the realizations of the particle system are independent and the last step follows from \Cref{ass:var}. We conclude using the central limit theorem with some sufficiently small confidence level that the statistical error satisfies
\[
\left|\mathcal{A}(M, P, N) - \E \el G^P(\bm X_N^P) \er \right| \lesssim M^{- \frac{1}{2}} P^{-1}.
\]
Finally, we can decompose the error as follows:
\begin{align*}
\left| \mathcal{A}(M, P, N) - \E \el g(z) \er \right| \leq & \;  \left| \mathcal{A}(M, P, N) - \E\el G^P(\bm X_N^P)\er \right| \\
& + \left| \E\el G^P(\bm X_N^P) - G^P(\bm X^P(T)) \er \right| \\
& + \left|\E\el G^P(\bm X^P(T)) -  g(z) \er \right| \\
\lesssim & M^{-\frac{1}{2}} P^{-1} + N^{-1} + P^{-2}.
\end{align*}
We continue with a rough estimation of the computational cost of $\mathcal{A}(M, P, N)$. Since we are average two times over the number of particles, once in the approximation of the measure, once outside, we price the number of particles twice. Multiplying with the number of samples and the cost of the discretization we obtain the total work
\[
\work(\mathcal{A}(M, P, N) = \Ocal(MNP^2).
\]

\begin{theorem}[Optimal work for single level]
Given that the number of samples $M$, particles $P$, and the step size $N^{-1}$ is chosen in an optimal fashion, the minimal computational work for the estimator $\mathcal{A}(M, P, N)$ satisfying a given error tolerance $\tol$ is of order
\[
\work(M, P, N) = \Ocal(\tol^{-3}).
\]
\end{theorem}

\begin{proof}
The proof boils down to solving a Lagrangian optimization problem with Lagrangian
\[
\mathcal{L}(M, N, P, \lambda) = MNP^2 + \lambda(M^{-\frac{1}{2}} P^{-1} + N^{-1} + P^{-2} - \tol).
\]
We compute
\begin{align*}
\ddi{\mathcal{L}}{M}(M, P, N, \lambda) & = NP^2 - \frac{1}{2} \lambda M^{-\frac{3}{2}} P^{-1}, \tag{I} \\
\ddi{\mathcal{L}}{P}(M, P, N, \lambda) & = 2 M N P - \lambda M^{-\frac{1}{2}} P^{-2} - 2 \lambda P^{-3} \tag{II} \\
\ddi{\mathcal{L}}{N}(M, P, N, \lambda) & = M P^{2} - \lambda N^{-2} \tag{III} \\
\ddi{\mathcal{L}}{\lambda}(M, P, N, \lambda) & = M^{-\frac{1}{2}} P^{-1} + N^{-1} + P^{-2} - \tol. \tag{IV}
\end{align*}
Enforcing $\nabla \mathcal{L} = 0$ yields with (III) that
\begin{align}\label{eq:lambda}
\lambda = M N^{2} P^{2}.
\end{align}
We substitute this into (I) to obtain
\begin{align}\label{eq:subst1}
N P^2 - \frac{1}{2}N^2 M^{- \frac{1}{2}} P = 0,
\end{align}
concluding
\begin{align}\label{eq:P}
P = \frac{1}{2} N M^{- \frac{1}{2}} = \Ocal(NM^{-\frac{1}{2}})
\end{align}
Substituting \Cref{eq:lambda} and \Cref{eq:P} into (II) yields
\begin{align*}
2 N^2 M^{\frac{1}{2}} - M^{\frac{1}{2}} N^2 - 2 M^{\frac{3}{2}} N = 0.
\end{align*}
We conclude
\[
N = \Ocal(M), \; P = \Ocal(M^{\frac{1}{2}}).
\]
Substituting this into (IV) yields
\[
M = \Ocal(\tol^{-1}), N = \Ocal(\tol^{-1}), P = \Ocal(\tol^{-\frac{1}{2}}).
\]
We conclude that we have a total work of
\[
\work(M, P, N) = \Ocal(\tol^{-3}).
\]
\end{proof}

\subsection{Multilevel}

For the multilevel estimator we again assume the same setting, that is \Cref{ass:euler}, \Cref{ass:weak_conv}, and \Cref{ass:var}. Moreover, we assume \Cref{ass:MLMC_var_ext} below, based on our observations, applying similar reasoning as in \cite{hajiali2017multilevel}.

\begin{conjecture}[Variance Decay]\label{ass:MLMC_var_ext}
We assume that for the extrapolated multilevel there are constants $c_1, c_2 \in \R$ such that the variances per level satisfy
\[
\var\el \prescript{Ext}{}{\Phi}^{\ell} - \prescript{Ext}{}{\Psi}^\ell \er \lesssim 2^{-2 \ell} \max(c_1 2^{-2 \ell}, c_2 2^{- 2 \ell}) = \Ocal(2^{-4 \ell}).
\]
\end{conjecture}

The extrapolated antithetic construction lifts the rate of the time difference contribution to a second order. Hence we obtain an $\Ocal(2^{-4 \ell})$ variance decay. The computational cost per sample remains the same.

\begin{conjecture}[Computational Work]
We assume that the computational work required for one realization satisfies
\[
\work\left[\prescript{Ext}{}{\Phi}^\ell - \prescript{Ext}{}{\Psi}^\ell \right] = \Ocal(2^{-3\ell}).
\]
\end{conjecture}

We formulate the optimal complexity of the multilevel estimator as a straightforward corollary of the MLMC complexity theorem in \cite{hajiali2017multilevel}.

\begin{corollary}[Optimal complexity]
The minimal computational work required for achieving an accuracy of $\tol$ with the estimator $\estMLMC$ is of order
\[
\work = \Ocal(\tol^{-2}).
\]
\end{corollary}

\section{Numerical Results}\label{sec:numres}

\subsection{Two Examples of McKean--Vlasov SDEs}

\subsubsection{An Ornstein-Uhlenbeck MV-SDE}\label{subsec:OU-SDE}
Let $\kappa, \sigma \in \R$ and $\xi$ be a square-integrable random variable. We consider the McKean--Vlasov SDE
\begin{align}\label{def:OU-MVSDE}
\begin{split}
dX(t) &= \kappa \int (y - X(t)) \mathcal{L}(X(t))(d y) d t + \sigma d W(t), \\
X(0) & = \xi.
\end{split}
\end{align}
This equation has the advantage that we can easily compute the first two moments of the solution at any time analytically. The proof boils down to a simple calculation.

\begin{proposition}\label{lem:OU_exact}
With the same assumptions as above, \Cref{def:OU-MVSDE} satisfies
\begin{enumerate}
\item $\E[X(t)] = \E[\xi]$ and
\item $\E[X(t)^2] = \frac{\sigma^2}{2\kappa} + \left(\E[\xi^2] - \frac{\sigma^2}{2 \kappa} \right) \exp(-2\kappa t)$.
\end{enumerate}
\end{proposition}

However, we see that the particle system in \Cref{eq:Particle-System}, here given by
\begin{align*}
d X_p^P(t) & = \xi_p + \kappa \left( \frac{1}{P} \sum_{j = 1}^P X_j^P(t) - X_p^{P}(t) \right) \d t + \sigma d W(t),
\end{align*}
yields in expectation an unbiased estimate of the true solution, regardless of the ensemble size $P$. Thus, the first moment is not a suitable candidate for numerical simulations of the particle system error. However, this is not to be expected for the second moment as we apply here a non-linearity to the solution before taking the expectation.

\subsubsection{Kuramoto Oscillator} \label{subsec:kuro}
This example is taken from \cite{hajiali2017multilevel}. We consider the McKean--Vlasov SDE
\begin{align}\label{def:kuramoto}
\begin{split}
d Z(t) &= \left(\nu + \int \sin(Z(t) - z) \mathcal{L}(Z(s))(d z) \right) d t + \sigma d W(t) \\
Z(0) & = \xi,
\end{split}
\end{align}
with $\sigma \in \R$ and $\xi$ a square integrable random variable and frequency $\nu \sim U([0, 1])$. Note that we have introduced here another random variable $\nu$ which we call the frequency. Let $\xi_i, W_i, \nu_i$ be i.i.d. copies of $\xi, W, \nu$, then the corresponding particle system reads for $1 \leq p \leq P$,
\begin{align*}
d X_p^P(t) & = \left(\nu + \frac{1}{P} \sum_{j = 1}^P \sin(X_p^P(t) - X_j^P(t)) \right) d t + \sigma d W_p(t) \\
Z(0) & = \xi_p.
\end{align*}
The Kuramoto model is often used to simulate, in the mean field limit, the synchronization of a large number of interacting agents. An example of this situation is population dynamics. For our purpose of observing the convergence rates we choose a Gauss function
\[
g(x) = \exp\left(-\frac{x^2}{2} \right),
\]
which is in particular smooth and due to the light tails is expected to have a dampening effect on the variance, i.e. we might need fewer samples than for other functions.

\subsection{Ornstein Uhlenbeck - Weak Error and Variance}
\FloatBarrier
In this section, we verify the convergence properties of our new estimator numerically for the Ornstein Uhlenbeck McKean--Vlasov SDE. For \Cref{ass:var} we simulate our particle system for a fixed time discretization parameter and a fixed number of samples, i.e. shifts. At the same time, we use a different number of QMC points, thus varying the size of the particle system. Finally, we estimate the variance of a $P$-point system as the sample variance ranging over all shifts. The result can be seen in \Cref{fig:OU_var}.

\begin{figure}
\includegraphics[width=0.87\textwidth]{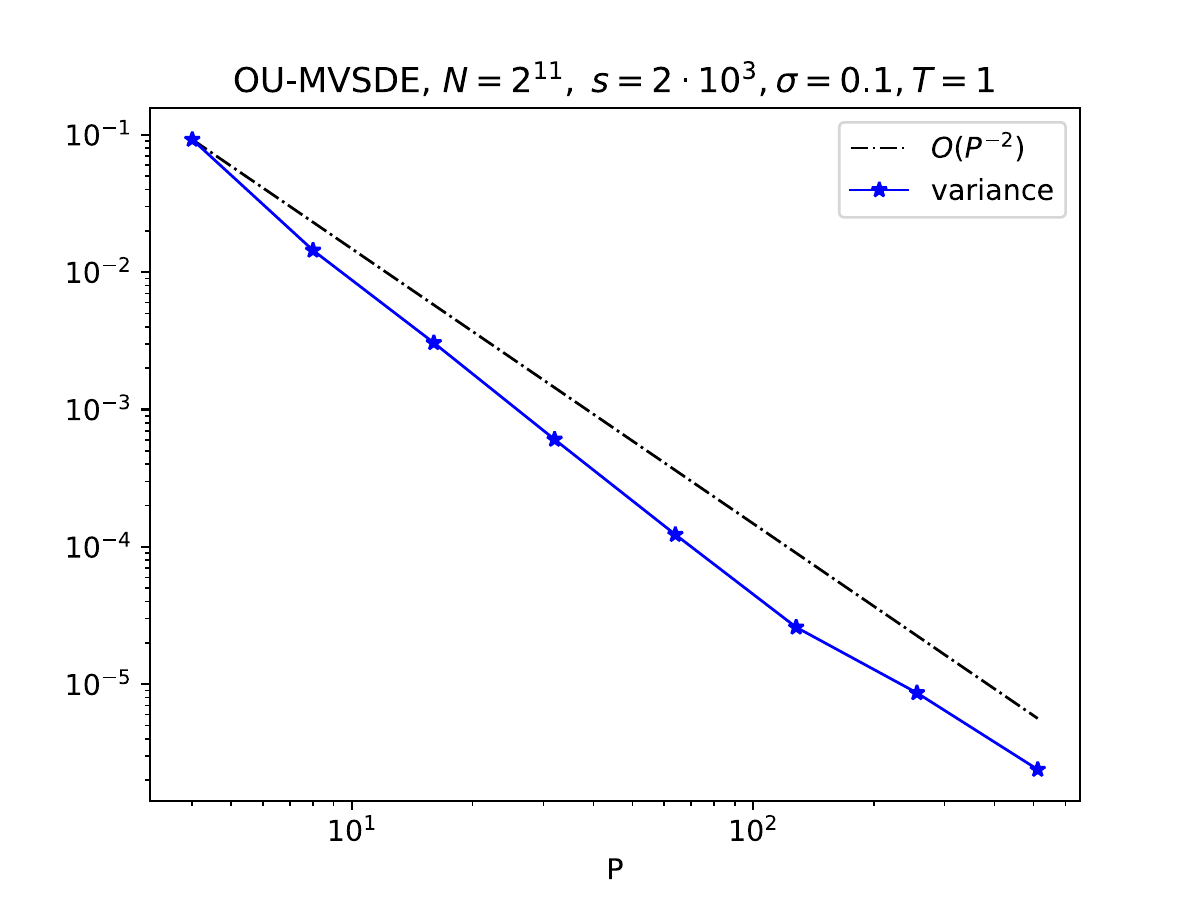}
\caption{Verification of \Cref{ass:var}}
\label{fig:OU_var}
\end{figure}

Next, we verify the weak convergence of our estimator for the Ornstein Uhlenbeck McKean--Vlasov SDE. We use a high number of samples as well as a small stepsize to reduce the other error sources and plot the exact error against the size of the particle system $P$, i.e. the number of QMC points. The exact error is calculated using the exact solution derived in \Cref{lem:OU_exact}, making this toy equation particularly useful to test the convergence of our new scheme.

\begin{figure}\label{fig:OU_conv}
\includegraphics[width=0.87\textwidth]{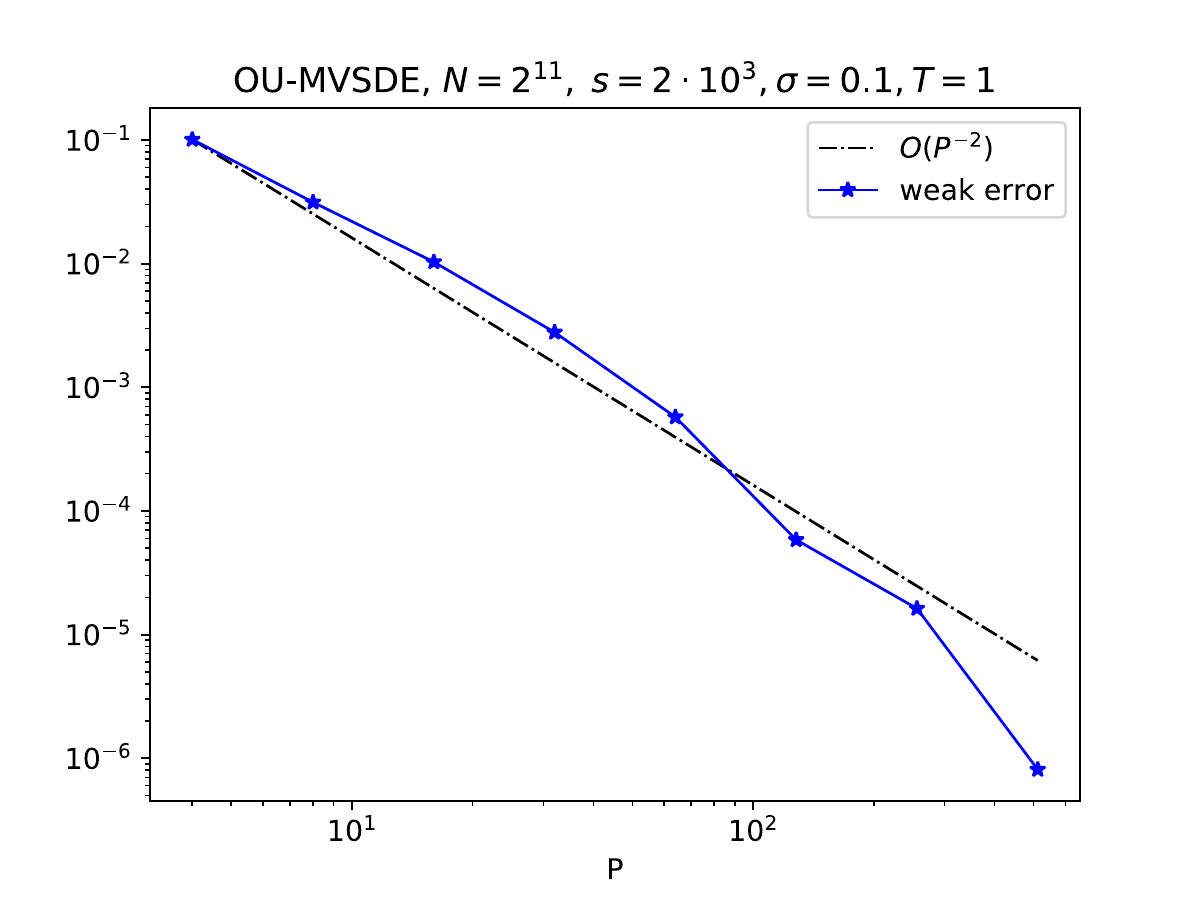}
\caption{Verification of \Cref{ass:weak_conv}}
\label{fig:OU_conv}
\end{figure}

\FloatBarrier

\subsection{Kuramoto oscillator}
\FloatBarrier

The Ornstein-Uhlenbeck equation has the great advantage of knowing the exact solution. However, being a linear model it does not tell us much about how good our method performs in a more general setting. Thus, we also employ the nonlinear Kuramoto model as introduced in \Cref{subsec:kuro}. The variance decay over the number of QMC points is simulated as described in the previous section.

\begin{figure}
\includegraphics[width=0.87\textwidth]{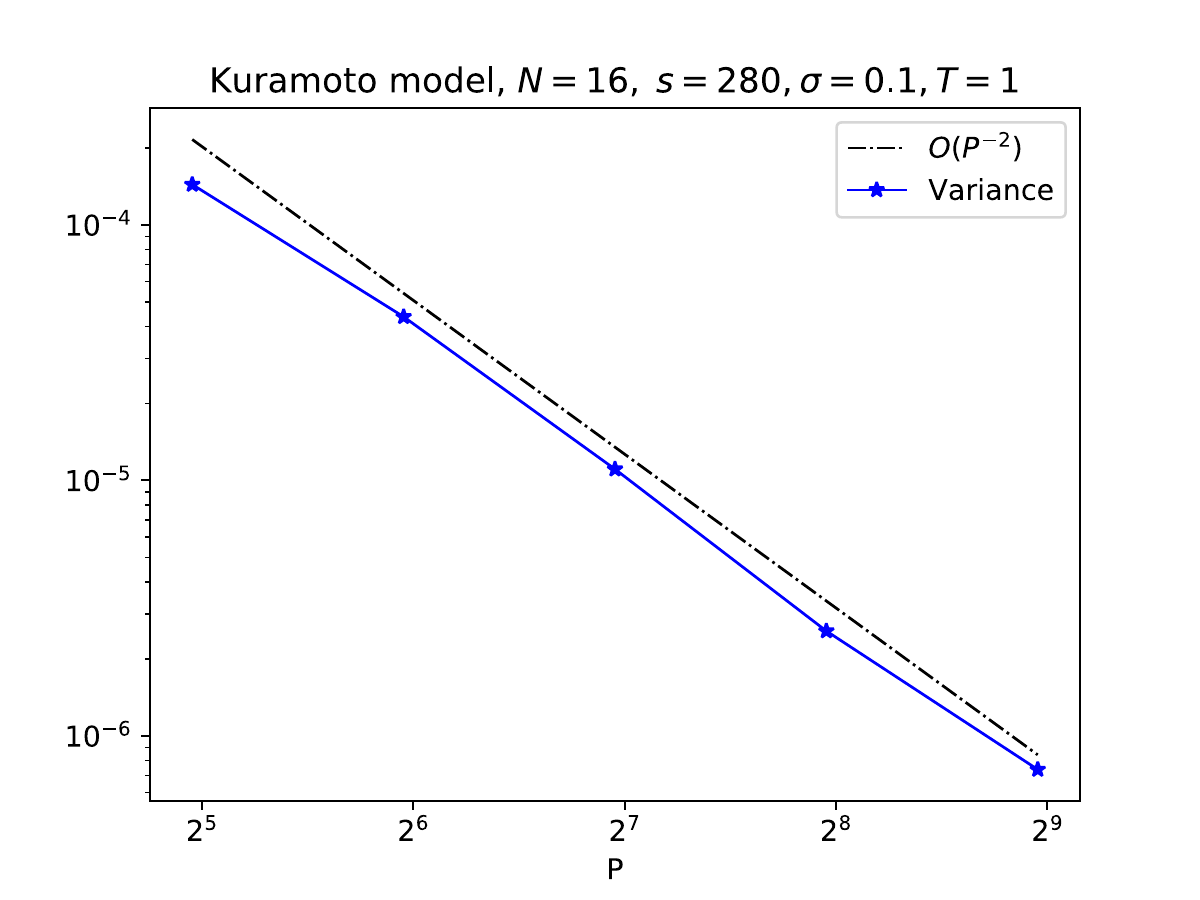}
\caption{Verification of \Cref{ass:var}}
\label{fig:kuro_var}
\end{figure}

For the weak convergence, we make use of an expensive reference solution that was computed using the standard particle system approximation, whose convergence to the mean-field limit is well established. As the reader can see we observe in \Cref{fig:kuro_conv}, the error is upper bounded by $\Ocal(P^{-2})$.

\begin{figure}
\includegraphics[width=0.87\textwidth]{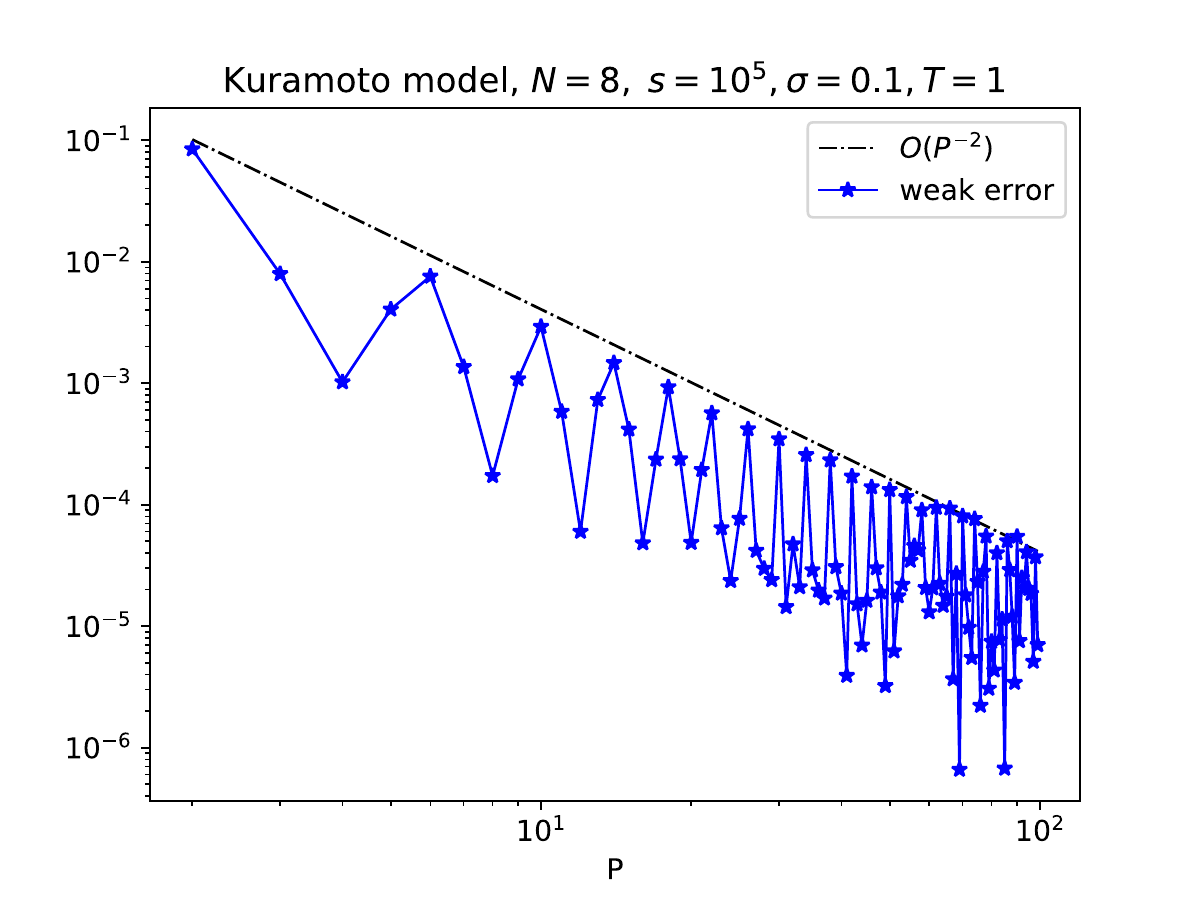}
\caption{Verification of \Cref{ass:weak_conv}}
\label{fig:kuro_conv}
\end{figure}

\FloatBarrier
\subsection{Kuramoto oscillator -- Hierarchical Variance Decay}
\FloatBarrier

Here we plot the variance decay along the level number for the multilevel construction. We can clearly see a rate of $\Ocal(2^{4\ell})$ in \Cref{fig:MLMC_var_kuro}.

\begin{figure}
\includegraphics[width=0.87\textwidth]{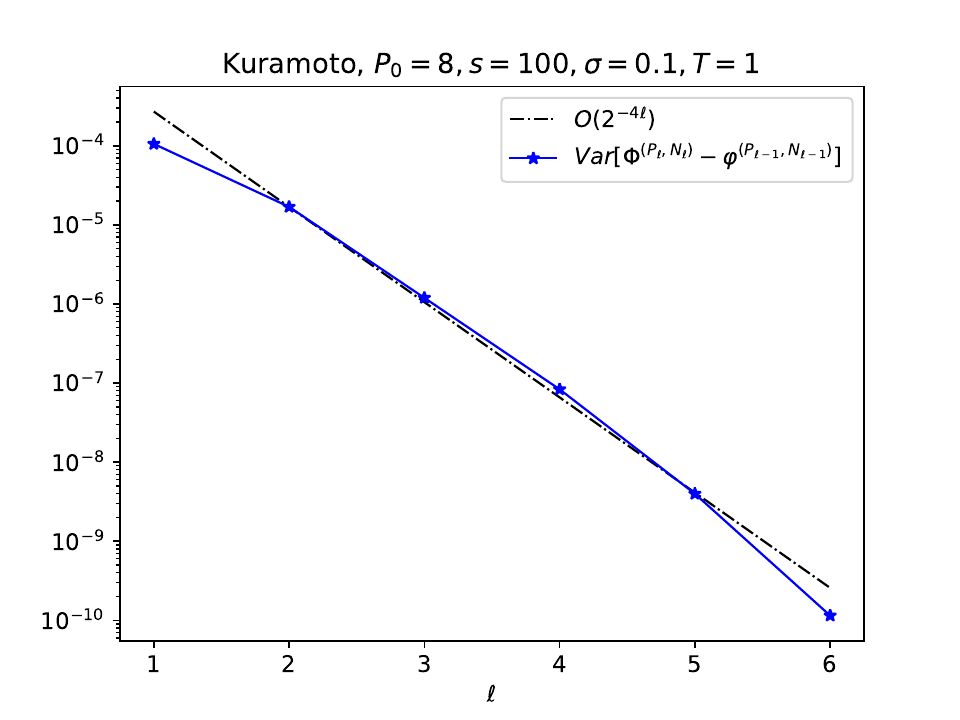}
\caption{Verification of \Cref{ass:MLMC_var_ext}}
\label{fig:MLMC_var_kuro}
\end{figure}

\FloatBarrier

\section{Conclusion}
This work has improved the weak and strong convergence rates of particle system approximations for McKean--Vlasov equations using QMC. For that purpose, we changed the stochastic input of a particle system with size $P$ from a standard Brownian Motion to a coupled system of $P$ Wiener processes. This was achieved by applying a Brownian Bridge technique to map $P$ QMC points to one path each. Shifting the $P$ points uniformly gave rise to the coupled Wiener processes which are fed to the particle system. This allows for a nice stochastic analytic interpretation of the estimator as an Euler-Maruyama approximation of an SDE driven by a $P$-dimensional martingale that can be dealt with using tools from stochastic calculus. Numerical results underpin the substantially improved efficiency when it comes to the approximation of mean-field limits in single-level and hierarchical approaches. Weak convergence is proven, but proofs of rates are so far only available in the zero-diffusion context. Certainly, further restrictions on the class of interaction kernels, drifts, and diffusions can open the door to use techniques as in \cite{liu2023nonasymptoticconvergenceratequasimonte} to do a concise analysis of convergence speed in the general SDE case along the same path of the proof of \Cref{thm:weak_convergence}. We leave this for future endeavors.

\section*{Acknowledgements}
Research reported in this
publication was supported by the King Abdullah University of Science and Technology (KAUST).

\printbibliography

\end{document}